\documentclass[11pt]{amsart}

\usepackage{latexsym,enumerate}
\usepackage{amssymb, xcolor,mathrsfs}
\usepackage{amsmath,amsthm,amsfonts,amssymb,latexsym, mathabx,lipsum}

\usepackage[colorlinks,pagebackref]{hyperref}
\usepackage[english]{babel}

\usepackage[utf8]{inputenc}

\newtheorem{thm}{Theorem}[section]
\newtheorem{lem}[thm]{Lemma}
\newtheorem{cor}[thm]{Corollary}
\newtheorem{prop}[thm]{Proposition}

\newtheorem{mainthm}{Theorem}

\theoremstyle{definition}
\newtheorem{define}[thm]{Definition}

\newtheorem{remark}{Remark}

\def\nor{\trianglelefteq\,}

\newcommand{\N}{\mathbb{N}}

\newcommand{\Irr}{{{\operatorname{Irr}}}}

\newcommand{\RNum}[1]{\uppercase\expandafter{\romannumeral #1\relax}}
\newcommand{\Al}{\textup{\textsf{A}}}
\newcommand{\Sy}{\textup{\textsf{S}}}

\newcommand{\cod}{\normalfont{\mbox{cod}}}
\newcommand{\cd}{\normalfont{\mbox{cd}}}

\newcommand{\PSL}{\normalfont{\mbox{PSL}}}

\newcommand{\PSU}{\normalfont{\mbox{PSU}}}

\newcommand{\irr}{\mathrm{Irr}}

\newcommand{\bZ}{{\mathbf Z}}

\newcommand{\bC}{{\mathbf{C}}}
\newcommand{\bG}{{\mathbf{G}}}

\newcommand{\bO}{{\mathbf{O}}}

\newcommand{\St}{{\mathsf {St}}}

\title[The codegree isomorphism problem for finite simple groups]
{The codegree isomorphism problem\\ for finite simple groups}

\author[N.\,N. Hung]{Nguyen N. Hung}
\address{Department of Mathematics, The University of Akron, Akron,
OH 44325, USA} \email{hungnguyen@uakron.edu}

\author[A. Moret\'{o}]{Alexander Moret\'{o}}
\address{Departamento de Matem\'aticas, Universidad de Valencia,
46100 Burjassot, Valencia, Spain} \email{alexander.moreto@uv.es}

\thanks{Part of this work was
done while the first author was visiting the Vietnam Institute for
Advanced Study in Mathematics (VIASM), and he thanks the VIASM for
its hospitality and financial support.}

\thanks{The research of the second author is supported by
Ministerio de Ciencia e Innovaci\'on (Grant PID2019-103854GB-I00
funded by MCIN/AEI/ 10.13039/501100011033) and Generalitat Valenciana CIAICO/2021/163.}


\begin{document}

\maketitle

\begin{abstract}
We study the codegree isomorphism problem for finite simple groups.
In particular, we show that such a group is determined by the
codegrees (counting multiplicity) of its irreducible characters. The
proof is uniform for all simple groups and only depends on the
classification by means of Artin-Tits' simple order theorem.
\end{abstract}


\section{Introduction}

It has been known for a while that every finite (quasi)simple group
is determined by its complex group algebra, i.e., by the set of
character degrees with multiplicities \cite{BNOT15}. (Recall that,
by the Wedderburn theorem, the multiset of character degrees
counting multiplicities of $G$ determines and is determined by its
complex group algebra $\mathbb{C}{G}\cong \bigoplus_{\chi\in\irr(G)}
\mathrm{Mat}_{\chi(1)}(\mathbb{C})$. Here, $\irr(G)$ is the set of
all irreducible characters of $G$ and $\mathrm{Mat}_n(\mathbb{C})$
is the $\mathbb{C}$-algebra of $n\times n$ matrices.) The version
without multiplicity (and therefore stronger) is a conjecture
proposed by B. Huppert in the late nineties \cite{Huppert}. This
conjecture has been extensively studied over the past two decades,
and progress has been made on a case-by-case basis using the
classification of finite simple groups
\cite{Huppert06,nt,Bessenrodt}. Further discussion on this and other
isomorphism problems of similar nature can be found in the recent
survey \cite{Margolis}.

This paper is concerned with \emph{the codegree isomorphism problem}
for finite simple groups. For a character $\chi$ of a finite group
$G$, the \emph{codegree} of $\chi$ is $\cod(\chi):=
|G:\ker(\chi)|/\chi(1)$. This notion was first introduced and
studied (in a slightly different form) by D. Chillag and M. Herzog
\cite{Chillag1} and D. Chillag, A. Mann, and O. Manz
\cite{Chillag-Mann}. It was later developed into the current form
used nowadays by G. Qian, Y. Wang, and H. Wei \cite{Qian} and has
been proved to have remarkable connections with the structure of
finite groups \cite{Qian,i,dl,Moreto,q21, cn}.

Recently, there has been great interest in the codegree analogue of
Huppert's conjecture (Problem~20.79 of the Kourovka Notebook
\cite{Khukhro}), to which we will refer to as \emph{the codegree
isomorphism conjecture}. From now on $\cod(G)$ denotes the set of
all the codegrees of irreducible characters of $G$.
\begin{equation}
  \tag{CIC}\label{eq:HCC}
  \parbox{\dimexpr\linewidth-12em}{%
    \strut
    \emph{Let $S$ be a finite nonabelian simple group and $G$ a
finite group. Then $\cod(G)=\cod(S)$ if and only if $G\cong S$.}
    \strut
  }
\end{equation}
The approach so far  to this problem is more or less similar to
Huppert's original method, and therefore, unfortunately, is still
case-by-case \cite{BahriAkh,Ahanjideh,gkl, gzy, LY22}.

Let $\cod(G)=\{c_1<c_2<...<c_k\}$ and $m_G(c_i)$ be the number of
irreducible characters of $G$ with codegree $c_i$. The multiset
$$C(G):=\{(c_i,m_G(c_i)):1\leq i\leq k\}$$ is called the
{group pseudo-algebra} of $G$, which can be viewed as the codegree
counterpart of the aforementioned complex group algebra $\mathbb{C}
G$. A natural weaker version of (\ref{eq:HCC}) asks whether $G$ and
$S$ must be isomorphic if $C(G)=C(S)$. We confirm this in our first
main result.

\begin{mainthm}\label{thm:main1}
Let $S$ be a finite simple group and $G$ a finite group. Then
$$C(G)=C(S) \text{ if and only if } G\cong S.$$
\end{mainthm}

The main novelty of this paper is a more uniform approach to these
codegree problems with as little case-by-case analysis as possible.
Our proof of Theorem \ref{thm:main1}, somewhat surprisingly, only
relies on the classification via the so-called \emph{simple order
theorem} (also known as the Artin-Tits theorem
\cite{Artin55,Kimmerle-et-al}), which states that two non-isomorphic
finite simple groups have the same order if and only if they are
either $PSL_4(2)$ and $PSL_3(4)$ or $\Omega_{2n+1}(q)$ and
$PSp_{2n}(q)$ for some $n\geq 3$ and odd $q$. This is perhaps the
first time that a result of this type is proved uniformly for all
simple groups.

There are two key ingredients in the proof of Theorem
\ref{thm:main1}. We find it remarkable that they admit strikingly
elementary proofs. The first provides a characterization of perfect
groups in terms of codegrees, see Theorem \ref{thm-char}. (Recall
that a group is perfect if it coincides with its derived subgroup.)
The second is an order-divisibility property involving character
codegrees of finite simple groups, see Theorem
\ref{thm-|S|divides|G|}.

In general the set of codegrees $\cod(G)$ of a finite group $G$ does
not determine its set of character degrees. However, this is not the
case when $G=S$ is simple. To see why, consider a prime divisor $p$
of $|S|$. Then $S$ has some nontrivial irreducible character of
$p'$-degree (by the squares of the degrees equation). As a result,
$S$ has some codegree divisible by $|S|_p$ -- the $p$-part of $|S|$,
and this is the largest power of $p$ dividing any codegree of $S$.
Therefore both the order $|S|$ and all the character degrees of $S$
are indeed determined by $\cod(S)$. This argument also proves that
if two simple groups have the same set of codegrees, then they have
the same order. Our next main result is a far stronger statement but
the proof requires the classification.

\begin{mainthm}\label{thm:main4}
Let $S$ and $H$ be finite simple groups such that $\cod(S)\subseteq
\cod(H)$. Then $S\cong H$.
\end{mainthm}

Theorem \ref{thm:main4} in fact is the first step in proving the
codegree isomorphism conjecture (\ref{eq:HCC}). Let $G$ be any
finite group and $H$ a simple group such that $\cod(G)=\cod(H)$, and
$N$ be a maximal normal subgroup of $G$ so that $S:=G/N$ is simple.
In order to prove $G\cong H$, one would first need to establish
$S\cong H$, under the assumption $\cod(S)\subseteq \cod(G)=\cod(H)$.
This is precisely what we do in Theorem~\ref{thm:main4} (see Theorem
\ref{thm:last}).

\begin{remark}
Using Theorem \ref{thm:main4}, we will prove in a subsequent paper
\cite{hmt} that (\ref{eq:HCC}) holds for all sporadic groups,
alternating groups, groups of Lie type of low rank, and, for the
first time in the degree/codegree isomorphism problem, groups of Lie
type of arbitrary rank over a field of prime order. Furthermore,
perhaps unexpectedly, we reduce (\ref{eq:HCC}) to a problem on
$p$-groups.
\end{remark}

\begin{remark}
The proof of Theorem \ref{thm:main4} is fairly complicated and
combines several techniques. In particular, it essentially utilizes
some deep results on the representation theory of finite simple
groups, including the classification of prime-power-degree
representations \cite{Malle-Zalesskii,BBOO01}, lower/upper bounds
for the largest degree of irreducible representations
\cite{VK85,LMT13}, and the existence of $p$-defect zero characters
\cite{Michler,Willems,Granville-Ono}. Along the way we prove an
effective and \emph{explicit} upper bound for the largest character
degree $b(S)$ of an exceptional group $S$ of Lie type (Theorem
\ref{lem:exceptional-b(S)}). (See Lemma~\ref{lem:bounds-for-b(S)}
for previous related work on symmetric and alternating groups
\cite{VK85} and classical groups \cite{LMT13}.)
\end{remark}

\begin{remark}
We need bounds for the largest character degree $b(S)$ in order to
control the behavior of $f(S):=|S|/b(S)$ -- the smallest nontrivial
codegree of $S$. While the relevance of the smallest (or low-degree
in general) characters of (quasi/almost)simple groups is well-known
in group representation theory (see \cite{MM21} for the latest
results), the smallest codegree had not been studied much before.
This invariant arises naturally in the proof of Theorem
\ref{thm:main4} (see Lemma~\ref{lem:fS geq fH}) and measures the
relative growth of $b(S)$ compared to $|S|$. Our proof would be much
simpler if one can show that $f$ is \emph{divisibly increasing}
among nonabelian simple groups, by which we mean that, if $S$ and
$H$ are nonabelian simple of different orders such that $|S|$
divides $|H|$, then $f(S)<f(H)$. We indeed confirm this phenomenon
in many cases, particularly when one of the two groups involved is
alternating (see Propositions \ref{prop:mixed case I},
\ref{prop:mixed case II}, and \ref{prop:alternating}).
\end{remark}

The layout of the paper is as follows. In Section
\ref{sec:prime-power-codegrees}, we prove some results on prime
character codegrees and provide a short proof of a theorem of Riese
and Schmid on prime-power codegrees. In Section
\ref{sec:codegrees-simple-gps} we discuss the order-divisibility
property  and its consequences involving character codegrees of
finite simple groups. Using the results in the preceding sections,
we prove Theorem \ref{thm:main1} in Section \ref{sec:theoremA}.
Results on bounding the largest character degree are presented in
Section \ref{sec:largest-degree}. Finally, the proof of Theorem
\ref{thm:main4} is carried out in Sections \ref{sec:theoremD-Lie},
\ref{sec:theoremD-mixed-case}, and
\ref{sec:theoremD-alternating-sporadic}.


\section{Prime-power codegrees}\label{sec:prime-power-codegrees}

In this section we prove some results on prime-power character
codegrees. These results show that, in contrast to character
degrees, there are significant restrictions on the structure of
groups with faithful irreducible characters of prime/prime-power
codegree.

We mainly follow the notation from \cite{Isaacs} for character
theory and \cite{Conway,Carter85} for finite simple groups.
Throughout, for a positive integer $n$ and a prime $p$, we write
$n_p$ to denote the maximal $p$-power divisor of $n$ and
$n_{p'}:=n/n_p$ to denote the maximal divisor not divisible by $p$
of $n$. Let $N\nor G$ and $\theta\in\irr(N)$. We write
$\irr(G|\theta)$ for the set of irreducible constituents of
$\theta^G$ and $\irr(G|N)$ for the set of irreducible characters of
$G$ whose kernels do not contain $N$. If $G$ is a group, $\pi(G)$ is
the set of primes that divide $|G|$. If $n$ is an integer, $\pi(n)$
is the set of primes that divide $n$, and if $\mathcal{S}$ is a set
of integers, then $\pi(\mathcal{S})$ is the set of primes that
divide some member of $\mathcal{S}$.  As usual,
$\cd(G):=\{\chi(1):\chi\in\irr(G)\}$ is the set of all irreducible
character degrees of $G$. Other notation will be recalled or defined
when necessary.

We begin by collecting some known facts on character codegrees that
we will use without explicit mention.

\begin{lem}\label{lem:1}
Let $G$ be a finite group and $\chi\in\Irr(G)$. The following hold:
\begin{enumerate}[\rm(i)]
\item
If $\chi$ is not the principal character, then $\cod(\chi)>\chi(1)$.
\item
If $N\trianglelefteq G$ and  $N\leq\ker(\chi)$, then the codegree of
$\chi$ as a character of $G$ coincides with the codegree of $\chi$
viewed as a character of $G/N$.
\item
If $N\trianglelefteq\trianglelefteq G$ and $\theta\in\Irr(N)$ lies
under $\chi$, then $\cod(\theta)$ divides $\cod(\chi)$.
\item
$\pi(G)=\pi(\cod(G))$.
\item
If $G$ is abelian, then $\cod(G)=o(G)$, where $o(G)$ is the set of
orders of the elements of $G$.
\end{enumerate}
\end{lem}

\begin{proof}
Part (i) is \cite[Lem. 2.1]{dl}. Parts (ii) and (iii) are contained
in \cite[Lem. 2.1]{Qian}, and part (iv) is  \cite[Lem. 2.4]{Qian}.
Now, we prove part (v). The inclusion $o(G)\subseteq\cod(G)$ follows
from \cite[Lem. 2.2]{dl}. Conversely, if $d\in\cod(G)$ there exists
$\chi\in\Irr(G)$ such that $d=|G:\ker(\chi)|$ (note that since $G$
is abelian, $\chi$ is linear). Since $G/\ker(\chi)$ is cyclic, we
conclude that $G$ has elements of order $d$.
\end{proof}

\subsection{Prime codegrees: characterizing perfect groups}

The goal in this subsection is to provide a characterization of
perfect groups in terms of the absence of prime codegrees.

\begin{thm}
\label{lem-p} Let $G$ be a finite group. Suppose that there exists
$\chi\in\Irr(G)$ faithful such that $\cod(\chi)=p$ is a prime
number. Then $G$ is either the cyclic group of order $p$ or a
Frobenius group with Frobenius kernel of order $p$.
\end{thm}

\begin{proof}
We argue by induction on $|G|$. Let $N\trianglelefteq G$ be minimal
such that there exists $\theta\in\Irr(N)$ lying under $\chi$ with
$\cod(\theta)=p$. Then
$$
\frac{|G|}{\chi(1)}=p=\cod(\theta)=\frac{[N:\ker(\theta)]}{\theta(1)},
$$
and we deduce that
$$
p=\cod(\chi)>\chi(1)=[G:N]|\ker(\theta)|\theta(1).
$$
In particular, $p$ does not divide any of the three factors in the
right hand side.

Suppose first that $N<G$. By the inductive hypothesis,
$N/\ker(\theta)$ is cyclic of order $p$ or a Frobenius group with
Frobenius kernel $K/\ker(\theta)$ of order $p$. In the latter case,
if $\lambda\in\Irr(K/\ker(\theta))$ lies under $\theta$ then
$\cod(\lambda)=p$. This contradicts the choice of $N$. (Note that
$K$ is normal in $G$ because $K$ is characteristic in $N$.)

Hence, we may assume that  $N/\ker(\theta)$ is cyclic of order $p$.
By Clifford's theorem the faithful character $\chi_N$ is a sum of
$G$-conjugates of $\theta$. Let $T$ be a complete set of
representatives in $G$ for these conjugates. By \cite[Lem.
2.21]{Isaacs}, the intersection of $\ker(\theta^g)$, where $g$ runs
over $T$, is trivial. We conclude that $N$ embeds into the direct
product \[\prod_{g\in T} N/\ker(\theta^g).\]  Each of the direct
factors has order $p$, and so $N$ is an elementary abelian
$p$-group. Since $p$ does not divide $|\ker(\theta)|$, we conclude
that $N$ is cyclic of order $p$. As $\theta$ is linear and
$\chi(1)=|G|/p$, we now have $\chi=\theta^G$. It follows that $G$ is
a Frobenius group with kernel $N$, as desired.

Now, we consider the case $N=G$. Let $M$ be a maximal normal
subgroup of $G$. Since $N=G$ the codegree of any irreducible
character of $N$ lying under $\chi$ is $1$. This means that
$\chi_M=\chi(1)\mathbf{1}_M$. But $\chi$ is faithful, so we deduce
that $M=1$ and $G$ is simple. If $G$ is abelian, then it is the
cyclic group of order $p$.  If $G$ is not abelian, then
$|G|/p=\chi(1)<\sqrt{|G|}$ and it follows that $|G|<p^2$.  By
Sylow's theorems, it follows that $G$ has a normal Sylow
$p$-subgroup. This contradiction completes the proof.
\end{proof}

The following consequence of Theorem \ref{lem-p} is already
mentioned in the introduction.

\begin{thm}
\label{thm-char} A finite nontrivial group $G$ is perfect if and
only if $G$ does not have any prime character codegree.
\end{thm}

\begin{proof}
By Lemma \ref{lem-p}, if $G$ has an irreducible character $\chi$ of
prime codegree then $G/\ker(\chi)$ is solvable. In particular, $G$ is
not perfect. Conversely, if $G$ is not perfect, then the abelian
group $G/G'$ has some irreducible character of prime codegree.
\end{proof}

\subsection{Prime power codegrees: the Riese-Schmid theorem}

Chillag and Herzog proved in \cite[Thm. 1]{Chillag1} that a simple
group does not possess nontrivial irreducible characters of prime
power codegree. The proof relied on a case by case analysis of the
simple groups, using the fact that, most of the times, they have
$p$-blocks of defect zero. This was generalized by Riese and Schmid
in \cite[Cor. 3]{RS} to quasisimple groups, using also block theory
and the classification. We offer a short proof of this result that
only depends on an easy consequence of the classification, which is
due to W. Kimmerle, R. Lyons, R. Sandling, and D.\,N. Teague
\cite[Thm. 3.6]{Kimmerle-et-al}:

\begin{lem}
\label{lem:aa} For every finite simple group $S$ and  prime $p$,
$|S|<(|S|_{p'})^2$.
\end{lem}

The following is a restatement of \cite[Lem. 1]{RS} in the language
of codegrees.

\begin{lem}
\label{lem-rs} Let $p$ be a prime. Let $G$ be a finite group and
$\chi\in\Irr(G)$ faithful. Then $\cod(\chi)$ is a power of $p$ if
and only if $\chi$ is induced from a Sylow $p$-subgroup of $G$.
\end{lem}

\begin{thm}
\label{thm-quasi}
A quasisimple group $G$ does not possess nonprincipal characters of
prime power codegree.
\end{thm}

\begin{proof}
Suppose that there exists $\mathbf{1}_G\neq\chi\in\Irr(G)$ of
$p$-power codegree. Let $K:=\ker(\chi)$ and note that $K\leq\bZ(G)$.
By Lemma \ref{lem-rs}, $\chi$ is induced from a Sylow $p$-subgroup
of $G/K$. Therefore, $$\chi(1)\geq|G/K|_{p'}.$$

By Lemma \ref{lem:aa}, we know that
$$|G/\bZ(G)|<(|G/\bZ(G)|_{p'})^2\leq(|G/K|_{p'})^2.$$ Hence, by
\cite[Cor. 2.30]{Isaacs},
$$
\chi(1)\leq|G:\bZ(G)|^{1/2} <|G/K|_{p'},
$$
which violates the inequality above.
\end{proof}


The next result is a
restatement in terms of character codegrees of Theorem~B of
\cite{RS}. The proof in \cite{RS} uses Brauer's first and third main
theorems.
Recall that if a group $G$ has trivial $p'$-core
$\bO_{p'}(G)$, then it is defined to be $p$-constrained if the
$p$-core $\bO_p(G)$ contains its centralizer.

\begin{thm}[Riese-Schmid]
\label{thm-rs} Let $G$ be a finite group and let $p$ be a prime.
Suppose that $\chi\in\Irr(G)$ is faithful of $p$-power codegree.
Then $\bO_{p'}(G)=1$ and $G$ is $p$-constrained.
\end{thm}

\begin{proof}
By Theorem \ref{thm-quasi}, we know that $G$ is not simple. Let $N$
be a minimal normal subgroup of $G$ and let $\theta\in\Irr(N)$ lying
under $\chi$. Since $\theta$ has $p$-power codegree (by Lemma
\ref{lem:1}(iii)) and $\cod(\theta)>1$ (note that since $\chi$ is
faithful, $\theta\neq\mathbf{1}_N$), we deduce that $N$ is either an
elementary abelian $p$-subgroup or a direct product of nonabelian
simple groups of order divisible by $p$. In particular,
$\bO_{p'}(G)=1$.

We claim that $N$ is an elementary abelian $p$-group. Suppose that
$N=S_1\times\cdots\times S_t$, with $S_i\cong S$ for some nonabelian
simple group $S$ of order divisible by $p$.  We wish to reach a
contradiction. Since $\theta\neq\mathbf{1}_N$, there exists a
nonprincipal $\psi\in\Irr(S_i)$ lying under $\theta$ for some $i$.
Note that $\cod(\psi)$ is a power of $p$ and this contradicts
Theorem \ref{thm-quasi}. The claim follows.

Write $P:=\bO_p(G)$ and $C:=\bC_G(P)$. Note that $C\cap P=\bZ(P)$
and $\bO_{p'}(C)=1$ (because $\bO_{p'}(G)=1$).  We want to see that
$C\leq P$. Assume not. Take $K$ subnormal in $G$ such that
$\bZ(P)\leq K\leq C$ and $K/\bZ(P)$ is simple. Since $\bO_{p'}(C)=1$
and $G$ does not have nonabelian minimal normal subgroups, we
conclude that $K'$ is quasisimple.

Now, take $\gamma\in\Irr(K')$ lying under $\chi$. Again, we have that
$\gamma$ is not principal and $\cod(\gamma)$ is a $p$-power. This
contradicts Theorem \ref{thm-quasi}.
\end{proof}

We end this section with a variation of Theorem \ref{thm-quasi}.

\begin{thm}
\label{lem-cf} Let $G$ be a finite group. Suppose that $p$ is a
prime and  $\chi\in\Irr(G)$ is faithful of $p$-power codegree.  Then
$\cod(\chi)$ exceeds the $p$-part of the product of the orders of
the nonabelian composition factors in a composition series of $G$.
In particular, if $K/L$ is a non-abelian chief factor of $G$, then
$\cod(\chi)>|K/L|_p$.
\end{thm}

\begin{proof}
Let $n$ be the product of the orders of the non-abelian composition
factors in a composition series of $G$. Using again that $\chi$ is
induced from a Sylow $p$-subgroup and  \cite{Kimmerle-et-al}, we
have
$$
\cod(\chi)>\chi(1)\geq|G|_{p'}\geq n_{p'}>n_p,
$$
as wanted.
\end{proof}


\section{An order-divisibility result for codegrees}\label{sec:codegrees-simple-gps}

The following order-divisibility result will be crucial in the
proofs of our main theorems.

\begin{thm}\label{thm-|S|divides|G|}
Suppose that $S$ is a finite simple group and $G$ a finite group
such that $\cod(S)\subseteq \cod(G)$. Then $|S|$ divides $|G|$.
\end{thm}

\begin{proof}
Let $d_1,...,d_k$ be all the degrees of nontrivial irreducible
characters of $S$, and let $m_i$ ($1\leq i\leq k$) be the number of
those characters of degree $d_i$. By the assumption, for each $i$,
there exists $\chi_i\in\irr(G)$ such that
\[
\frac{|S|}{d_i}=\frac{[G:\ker(\chi_i)]}{\chi_i(1)}=\frac{|G|}{\chi_i(1)|\ker(\chi_i)|}.
\]
It follows that
\[
\sum_{i=1}^k \frac{m_i d_i^2}{|S|^2}= \sum_{i=1}^k \frac{m_i
\chi_i(1)^2|\ker(\chi_i)|^2}{|G|^2},
\]
and thus
\[
\frac{\sum_{i=1}^k m_i
\chi_i(1)^2|\ker(\chi_i)|^2}{|G|^2}=\frac{\sum_{i=1}^k m_i
d_i^2}{|S|^2}=\frac{|S|-1}{|S|^2}.
\]
Therefore $|S|^2$ divides $|G|^2(|S|-1)$, and the theorem follows.
\end{proof}

We record some consequences of Theorem~\ref{thm-|S|divides|G|} that
will be needed in subsequent sections.

\begin{cor}\label{thm-|S|divides|G|2}
Suppose that $S$ and $H$ are finite simple groups such that $\cod(S)=
\cod(H)$. Then $|S|=|H|$.
\end{cor}

\begin{proof}
This directly follows from Theorem \ref{thm-|S|divides|G|}. (Or, alternatively, from the comment in the paragraph that precedes the statement of Theorem \ref{thm:main4}.)
\end{proof}


\begin{lem}\label{lem-x=H/S}
Suppose that $S$ and $H$ are finite nonabelian simple groups such
that $\cod(S)\subseteq \cod(H)$. Let $x:=|H|/|S|$. Then $x\in
\mathbb{N}$ and $dx\in \cd(H)$ for every $1\neq d\in\cd(S)$.
\end{lem}

\begin{proof}
We know that $x\in \mathbb{N}$ by Theorem~\ref{thm-|S|divides|G|}. For
each $1\neq d\in\cd(S)$, we have $|S|/d\in \cod(H)$, and thus there
exists some $\chi\in\irr(H)$ such that $|S|/d=|H|/\chi(1)$, implying
that $\chi(1)=dx$, as claimed.
\end{proof}

\begin{lem}\label{lem-chi-p-power}
Let $S$ and $H$ be finite simple groups of Lie type. Suppose that
the defining characteristic of $H$ is $p$ and $\cod(S)\subseteq
\cod(H)$. Then $|S|_{p'}=|H|_{p'}$ and there exists $\chi\in\irr(S)$
such that $\chi(1)=|S|_p$.
\end{lem}

\begin{proof}
We first observe that $|S|$ is divisible by $p$ because otherwise
every codegree of $S$ is not divisible by $p$ but the only
nontrivial codegree of $H$ not divisible by $p$ is
$|H|_{p'}=|H|/\mathrm{St}_H(1)$.

By \cite{Michler,Willems}, $S$ has an irreducible character, say
$\chi$, of $p$-defect $0$, so that $\cod(\chi)$ is coprime to $p$.
Therefore we have
\[
|S|/\chi(1)= |H|_{p'}.
\]
It follows from Theorem~\ref{thm-|S|divides|G|} that
$\chi(1)|H|_{p'}=|S|$ divides $|H|$, implying that
$|S|_{p'}=|H|_{p'}$ and $\chi(1)=|S|_p$.
\end{proof}

Remark that when $p\geq 5$, the above result holds for all $S$. This
is because, in such case, irreducible $p$-defect zero characters
still exist in alternating groups (by \cite{Granville-Ono}) and in
sporadic simple groups \cite{Conway}.

\begin{lem}\label{lem-chi-p-power-2}
Let $S$ be a finite simple group of Lie type and $H$ a finite
nonabelian simple group. Suppose that $\cod(S)\subseteq \cod(H)$ and
there are primes $p\neq r$ such that $H$ has a unique character
degree divisible by each $|H|_{p}$ and $|H|_r$. Then $|S|=|H|$ and
$\cd(S)\subseteq \cd(H)$.
\end{lem}

\begin{proof}
Repeating the arguments in the proof of Lemma~\ref{lem-chi-p-power},
we have $|S|_{p'}=|H|_{p'}$ and $|S|_{r'}=|H|_{r'}$, implying that
$|S|=|H|$ and $\cd(S)\subseteq \cd(H)$.
\end{proof}


\section{Group pseudo-algebras of simple groups: Theorem
\ref{thm:main1}}\label{sec:theoremA}

In this section we prove Theorem \ref{thm:main1}, using the results
in the preceding sections.

Theorem \ref{thm-|S|divides|G|} is useful in proving results
concerning codegrees of finite simple groups. One of them is the
next theorem, whose proof makes use of the simple order theorem.
Recall that the simple order theorem  asserts that two
non-isomorphic finite simple groups have the same order if and only
if they are either $PSL_4(2)$ and $PSL_3(4)$ (of order $20,160$) or
$\Omega_{2n+1}(q)$ and $PSp_{2n}(q)$ (odd-dimensional orthogonal and
symplectic groups, of order $(1/2)q^{n^2}\prod_{i=1}^n (q^{2i}-1)$)
for some $n\geq 3$ and odd $q$ (see \cite[Thm. 5.1]{Kimmerle-et-al}
for instance). It was proved by E. Artin \cite{Artin551,Artin55} in
the fifties for known families of simple groups at the time, and
completed by J. Tits for the remaining families discovered later on
(see \cite{Kimmerle-et-al}). Artin's method is to consider certain
invariants associated to (orders of) simple groups that can be
computed explicitly and are able to distinguish the groups easily.
Therefore, the simple order theorem currently relies on the
classification of finite simple groups.

\begin{thm}\label{cor-cod(S)=cod(H)}
Suppose that $S$ and $H$ are finite simple groups such that
$\cod(S)= \cod(H)$. Then $S\cong H$.
\end{thm}

\begin{proof} The statement is trivial when both of $S$ and $H$ are abelian.
If one of the two groups is nonabelian, then, by Corollary
\ref{thm-char}, so is the other. So assume that both $S$ and $H$ are
nonabelian. By Theorem \ref{thm-|S|divides|G|2}, we have $|S|=|H|$
and hence it follows that $\cd(S)=\cd(H)$. Assume to the contrary
that $S$ is not isomorphic to $H$.

By the simple order theorem, we have
$$\{S,H\}=\{PSL_4(2),PSL_3(4)\}$$ or
$$\{S,H\}=\{\Omega_{2n+1}(q),PSp_{2n}(q)\}$$ for some odd prime power
$q=p^\ell$ and $n\geq 3$. The former case is eliminated using
\cite{Conway}, so we just need to show that
$\cd(\Omega_{2n+1}(q))\neq \cd(PSp_{2n}(q))$ for indicated $n$ and
$q$.

By Lusztig's classification of ordinary characters of finite groups
of Lie type (see \cite[Chapter 13]{DM}), irreducible characters of
$G:=Sp_{2n}(q)$ are parameterized by pairs $((s),\psi)$ where $(s)$
is the conjugacy class of a semisimple element $s\in
G^*:=SO_{2n+1}(q)$ and $\psi$ is a unipotent character of the
centralizer $C:=\mathbf{C}_{G^\ast}(s)$. Moreover, the degree of the
character $\chi_{((s),\psi)}$ associated to $((s),\psi)$ is
\[
\chi_{((s),\psi)}(1)=[G^\ast:C]_{p'}\psi(1).
\]
Let $\alpha\in\{\pm 1\}$ such that $4\mid (q^m-\alpha)$ and consider
a semisimple element $s\in G^*$ with spectrum $\{1,-1,...,-1\}$ such
that $C\cong GO_{2n}^\alpha$ (see \cite[Lem. 2.2]{Nguyen10}). Such
$s$ will then belong to $\Omega_{2n+1}(q)=[G^*,G^*]$, implying that
the semisimple character $\chi_{((s),\mathbf{1}_C)}$ associated to
the pair $((s),\mathbf{1}_C)$ is trivial on
$\mathbf{Z}(Sp_{2n}(q))$, by \cite[Lem. 4.4]{Navarro-Tiep13}. We
therefore have an irreducible character of $PSp_{2n}(q)$ of degree
\[\chi_{(s)}(1)=(|SO_{2n+1}(q)|/|GO_{2n}^\alpha|)_{p'}=(q^n+\alpha)/2.\]

To see that $\cd(PSp_{2n}(q)) \neq \cd(\Omega_{2n+1}(q))$ (for odd
$q$ and $n\geq 3$), it is enough to show that $(q^n+\alpha)/2$ is
not character degree of $\Omega_{2n+1}(q)$. By \cite[Thm.
6.1]{Tiep-Zalesskii96}, under our assumptions on $n$ and $q$ and the
additional condition $(n,q)\neq (3,3)$, the minimal (nontrivial)
irreducible character of $Spin_{2n+1}(q)$ has degree
\[
d(Spin_{2n+1}(q))=\left\{\begin{array}{ll} (q^{2n}-1)/(q^2-1)& \mathrm{ if }\ q\geq 5 \\
(q^n-1)(q^n-q)/2(q+1) & \mathrm{if}\ q=3,\end{array} \right.
\]
(For the definition  of classical groups of various isogeny types,
including the odd-dimensional spin groups, we refer the reader to
\cite[p. 40]{Carter85}.) Note that $Spin_{2n+1}(q)$ is a central
extension of $\Omega_{2n+1}(q)$ and so every character degree of
$\Omega_{2n+1}(q)$ is one of $Spin_{2n+1}(q)$. It is now easy to
check that $d(Spin_{2n+1}(q))>(q^n+\alpha)/2$ for $n\geq 3$. For the
remaining case $(n,q)=(3,3)$, we note that $d(Spin_{2n+1}(q))=27$,
which is still greater than $(q^n+\alpha)/2=13$.
\end{proof}

Certainly, one has to do much more work to relax the hypothesis in
Theorem \ref{cor-cod(S)=cod(H)} to $\cod(S)\subseteq \cod(H)$; that
is, to obtain Theorem \ref{thm:main4}.


\begin{lem}\label{lem:O2n+1andPSp}
Let $n\geq 3$ and $q$ be an odd prime power. Then
$\cd(\Omega_{2n+1}(q))\nsubseteq \cd(PSp_{2n}(q))$ and
$\cd(PSp_{2n}(q)) \nsubseteq \cd(\Omega_{2n+1}(q))$.
\end{lem}

\begin{proof}
We have seen in the proof of Theorem \ref{cor-cod(S)=cod(H)} that
$PSp_{2n}(q)$ possesses an irreducible character of degree
$(q^n+\alpha)/2$, where $\alpha\in\{\pm 1\}$ such that $4 \mid
(q^n-\alpha)$, and furthermore $(q^n+\alpha)/2\notin
\cd(\Omega_{2n+1}(q))$. Therefore, it suffices to show
$cd(\Omega_{2n+1}(q))\nsubseteq \cd(PSp_{2n}(q))$.

We claim that both \[(q^{2n}-1)/(q^2-1) \text{ and }
q(q^{2n}-1)/(q^2-1)\] are elements of $\cd(\Omega_{2n+1}(q))$ for
$q$ odd. Let $G:=Spin_{2n+1}(q)$, the universal cover of
$\Omega_{2n+1}(q)$. The dual group $G^\ast$ of $G$ (in the sense of
\cite[Def. 13.10]{DM}) is the projective conformal symplectic group
$PCSp_{2n}(q)$, which is the quotient of $\widetilde{G}=CSp_{2n}(q)$
by its center $\bZ(\widetilde{G})\simeq C_{q-1}$. Consider a
semisimple element $s\in \widetilde{G}$ with spectrum
$Spec(s)=\{-1,-1,1,...,1\}$ and
\[\mathbf{C}_{\widetilde{G}}(s)\cong (Sp_{2}(q)\times
Sp_{2n-2}(q))\cdot C_{q-1}\] (see \cite[Lem. 2.4]{Nguyen10}). Let
$s^\ast$ be the image of $s$ under the natural homomorphism from
$\widetilde{G}$ to $G^\ast$. It is easy to see that, by the choice
of $s$, $\mathbf{C}_{\widetilde{G}}(s)$ is the complete inverse
image of $\mathbf{C}_{G^\ast}(s^\ast)$ under this homomorphism, and
thus
$\mathbf{C}_{G^\ast}(s^\ast)=\mathbf{C}_{\widetilde{G}}(s)/\mathbf{Z}(\widetilde{G})$
and
\[[G^\ast:\mathbf{C}_{G^\ast}(s^\ast)]_{p'}=[\widetilde{G}:\mathbf{C}_{\widetilde{G}}(s)]_{p'}
=\frac{|Sp_{2n}(q)|_{p'}}{|Sp_2(q)|_{p'}|Sp_{2n-2}(q)|_{p'}}=\frac{q^{2n}-1}{q^2-1},\]
where $p$ is the defining characteristic of $G$.

Consider the canonical homomorphism $f: Sp_2(q)\times
Sp_{2n-2}(q)\hookrightarrow \mathbf{C}_{\widetilde{G}}(s)
\rightarrow \mathbf{C}_{G^\ast}(s^\ast)$. Using \cite[Prop.
13.20]{DM}, we know that unipotent characters of $Sp_2(q)\times
Sp_{2n-2}(q)$ are of the form $\theta \circ f$ where $\theta$ runs
over the unipotent characters of $\mathbf{C}_{G^\ast}(s^\ast)$. In
particular, as $Sp_{2}(q)\cong SL_2(q)$ has unipotent characters of
degrees 1 and $q$, $\mathbf{C}_{G^\ast}(s^\ast)$ has two unipotent
characters of degree $1$ and $q$ as well. By the conclusion of the
previous paragraph, the Lusztig series $\mathcal{E}(G,(s^\ast))$ of
$G$ associated to the conjugacy class $(s^\ast)$ of $G^\ast$
contains two characters of degrees $(q^{2n}-1)/(q^2-1)$ and
$q(q^{2n}-1)/(q^2-1)$.

Note that $s^\ast \in PSp_{2n}(q)=[G^\ast,G^\ast]$ and
$|\mathbf{Z}(G)|=|G^\ast/(G^\ast)'|=2$, and therefore every
character in the Lusztig series $\mathcal{E}(G,(s^\ast))$ restricts
trivially to $\mathbf{Z}(G)$ (see \cite[Lem. 4.4]{Navarro-Tiep13}),
and so can be viewed as a character of
$G/\mathbf{Z}(G)=\Omega_{2n+1}(q)$. The claim is completely proved.

Suppose first that $q>3$ and assume to the contrary that
$\cd(\Omega_{2n+1}(q))\subseteq \cd(PSp_{2n}(q))$. Then we have
$(q^{2n}-1)/(q^2-1)\in \cd(PSp_{2n}(q))$, so that
$(q^{2n}-1)/(q^2-1)\in \cd(Sp_{2n}(q))$. Let
$\chi\in\irr(Sp_{2n}(q))$ such that $\chi(1)=(q^{2n}-1)/(q^2-1)$.
Now, as $q>3$, we have $(q^{2n}-1)/(q^2-1)<(q^{2n}-1)/2(q+1)$, and
therefore, by the classification of irreducible characters of
$Sp_{2n}(q)$ of degrees up to $(q^{2n}-1)2(q+1)$ (\cite[Thm.
5.2]{Tiep-Zalesskii96}), we deduce that $\chi$ must be one of the
Weyl characters of degree $(q^n\pm 1)/2$ or the minimal unipotent
one of degree $(q^n-1)(q^n-q)/2(q+1)$. A simple check reveals that
none of these degrees matches the degree of $\chi$.

Now we suppose $q=3$. (In such case, $(q^{2n}-1)/(q^2-1)$ is indeed
a character degree of $PSp_{2n}(q)$.) As the case of $\Omega_7(3)$
and $PSp_6(3)$ can be checked directly, we suppose furthermore that
$n\geq 4$. We then have
$q(q^{2n}-1)/(q^2-1)<(q^{2n}-1)(q^{n-1}-q)/2(q^2-1)$. Examining the
degrees up to $(q^{2n}-1)(q^{n-1}-q)/2(q^2-1)$ of irreducible
characters of $Sp_{2n}(q)$ available in \cite[Cor. 4.2]{Nguyen10},
we observe that none of them is equal to $q(q^{2n}-1)/(q^2-1)$. We
have shown that $q(q^{2n}-1)/(q^2-1)\notin \cd(Sp_{2n}(q))$, and
therefore, by the above claim, $\cd(\Omega_{2n+1}(q))\nsubseteq
\cd(PSp_{2n}(q))$, as desired.
\end{proof}

We can now prove our first main Theorem \ref{thm:main1}, which in
fact follows from the following slightly stronger result. If $G$ and
$H$ are groups we say that $C(G)\subseteq C(H)$ if
$\cod(G)\subseteq\cod(H)$ and $m_G(c)\leq m_H(c)$ for every
$c\in\cod(G)$.

\begin{thm}\label{thm:main1repeated}
Let $H$ be a  finite simple group and $G$ a nontrivial finite group
such that $C(G)\subseteq C(H)$. Then $G\cong H$.
\end{thm}

\begin{proof}
Suppose first that $H$ is abelian of prime order $p$. Then
$\cod(G)\subseteq\cod(H)=\{1,p\}$. Therefore, by Lemma 2.4 of \cite{dl}, $G$ is an elementary
abelian $p$-group and since $k(G)\leq p$, we conclude that $G$ is
cyclic of order $p$, as wanted.


So we may assume that $H$ is nonabelian. By Corollary \ref{thm-char}
and the assumption $C(G)\subseteq C(H)$, we have that $G$ is
perfect. Let $N\nor G$ such that $S:=G/N$ is nonabelian simple. Now
$\cod(S)\subseteq \cod(G)\subseteq \cod(H)$ and therefore, by
Theorem \ref{thm-|S|divides|G|}, we have $|S|$ divides $|H|$.

Note that $C(S)\subseteq C(G)\subseteq C(H)$. Therefore there exists a subset
$I\subseteq \irr(H)\backslash \{\mathbf{1}_H\}$ and a bijection
$f:\irr(S)\backslash \{\mathbf{1}_S\} \rightarrow I$ such that
\[
\frac{|S|}{\chi(1)}=\frac{|H|}{f(\chi)(1)}
\]
for every $\chi\in \irr(S)\backslash \{\mathbf{1}_S\}$. It follows
that
\[
\sum_{\chi\in \irr(S)\backslash \{\mathbf{1}_S\}}
\frac{\chi(1)^2}{|S|^2} = \sum_{\psi\in I} \frac{\psi(1)^2}{|H|^2},
\]
and thus
\[
\frac{|S|-1}{|S|^2}\leq \frac{|H|-1}{|H|^2}.
\]
As the function $(x-1)/x^2$ is decreasing on $[2,\infty)$, we deduce
that $|S|\geq |H|$.

The conclusions of the last two paragraphs show that $|S|=|H|$. If
$S\cong H$ then $C(G/N)=C(S)=C(H)\supseteq C(G)$ and so $G/N$ and $G$ have
the same number of conjugacy classes, which is possible only when
$N=1$, and we are done.

So assume by contradiction that $S\ncong H$. Using again the simple
order theorem, we have $\{S,H\}=\{PSL_4(2),PSL_3(4)\}$ or
$\{S,H\}=\{\Omega_{2n+1}(q), PSp_{2n}(q)\}$ for some $n\geq 3$ and
odd $q$. For the former pair, using the character tables of both
$PSL_4(2)$ and $PSL_3(4)$ available in \cite{Conway}, one observes
that $7\in\cd(PSL_4(2)) \backslash \cd(PSL_3(4))$ and $63\in
\cd(PSL_3(4)) \backslash \cd(PSL_4(2))$, implying that none of
$\cod(PSL_3(4))$ and $\cod(PSL_2(4))$ contains the other, and this
violates the fact that $\cod(S)\subseteq \cod(H)$. The latter pair
was already handled in Lemma \ref{lem:O2n+1andPSp}.
\end{proof}


\section{The largest character degree of finite simple
groups}\label{sec:largest-degree}

Let $b(G)$ denote the largest degree of an irreducible character of
a finite group $G$. Recall from the Introduction that if $S$ is a
simple group, then $f(S):=|S|/b(S)$ is the smallest nontrivial
character codegree of $S $. The following elementary fact explains
the relevance of $f(S)$, and therefore $b(S)$.

\begin{lem}\label{lem:fS geq fH}
Let $S$ and $H$ be finite simple groups such that $\cod(S)\subseteq
\cod(H)$. Then $f(S)\geq f(H)$.
\end{lem}

\begin{proof}
The hypothesis implies that $f(S)\in\cod(S)\subseteq\cod(H)$. Since $f(H)$ is the smallest nontrivial member of $\cod(H)$, it follows that $f(S)\geq
f(H)$.
\end{proof}

Under the hypothesis of Lemma \ref{lem:fS geq fH}, we showed in
Theorem \ref{thm-|S|divides|G|} that $|S|$ divides $|H|$. We will see in
later sections that, in many cases, the two conditions $f(S)\geq
f(H)$ and $|S|$ divides $|H|$ are enough to force $S\cong H$, as
stated in Theorem \ref{thm:main4}.

Browsing through character tables of small-order simple groups in
\cite{Conway}, one notices that if $|H|$ is a multiple of $|S|$ and $|H|>|S|$, then $b(H)>b(S)$. However, it seems that the largest
character degree grows slower than the order -- that is,
$f(H)>f(S)$. This is not so easy to prove generally, but we do
confirm it in several cases, particularly when either $S$ or $H$ is
an alternating group (see Sections \ref{sec:theoremD-mixed-case} and
\ref{sec:theoremD-alternating-sporadic}).

We shall need effective (both lower and upper) bounds for the
largest degree of an irreducible character of simple groups. For
symmetric groups, asymptotic and explicit bounds were obtained by
A.\,M. Vershik and S.\,V. Kerov in \cite{VK85} which can be used to
derive the corresponding bounds for alternating groups. For a group
$S$ of Lie type in characteristic $p$, an obvious (and in fact very
tight!) lower bound for $b(S)$ is the degree $\St_S(1)=|S|_p$ of the
Steinberg character $\St_S$. When $S$ is of classical type, explicit
upper bounds have been worked out by M. Larsen, G. Malle, and P.\,H.
Tiep in \cite{LMT13}. Unfortunately, upper bounds for exceptional
groups achieved in \cite{LMT13} are only asymptotic and its proof
does not allow one to obtain an explicit bound. We obtain Theorem
\ref{lem:exceptional-b(S)} below that we believe will be useful in
other applications.

\begin{lem}\label{lem:order-bound}
Let $\bG$ be a simple algebraic group over the algebraic closure of
a finite field of order $q$ in characteristic $p$, $F:
\bG\rightarrow \bG$ a Steinberg endomorphism, and $G:= \bG^F$ be the
corresponding finite group of Lie type. Let $r$ be the rank of
$\bG$. Then
\[
(q-1)^r\cdot|G|_p\leq |G|_{p'}\leq q^r\cdot|G|_p.
\]
\end{lem}

\begin{proof}
Note that finite groups $\bG^F$ of the same isogeny type have the
same order, so we may work with $\bG$ being of simply-connected
type. The inequalities are then straightforward to verify using the
order formulas for finite groups of Lie type available in \cite[p.
xvi]{Conway}.
\end{proof}

\begin{thm}\label{lem:exceptional-b(S)}
Let $S$ be a simple exceptional group of Lie type defined over a
field of order $q$ in characteristic $p$. Then the following hold:
\begin{enumerate}[\rm(i)]
\item $b(S)<256|S|_p$.

\item If $q>2$, then $b(S)<26|S|_p$.
\end{enumerate}
\end{thm}

\begin{proof}
The Tits group can be verified easily, so we assume that $S\neq
{}^2F_4(2)'$. We then may find a simple algebraic group $\mathbf{G}$
of adjoint type and a Steinberg endomorphism
$F:\mathbf{G}\rightarrow \mathbf{G}$ such that $S=[G,G]$ where
$G:=\mathbf{G}^F$ (see \cite[Prop. 24.21]{malletesterman}). Clearly
it suffices to show that $b(G)<256 \St_G(1)$.

Let $(\mathbf{G}^\ast,F^\ast)$ be the dual pair of $(\mathbf{G},F)$,
so $\mathbf{G}^\ast$ will be the corresponding simple algebraic
group of simply connected type, and set
$G^\ast:=(\bG^\ast)^{F^\ast}$. As mentioned before, Lusztig's
classification on complex characters of finite groups of Lie type
implies that the set of irreducible complex characters of $G$ is
partitioned into Lusztig series $\mathcal{E}(G,(s))$ associated to
various conjugacy classes $(s)$ of semisimple elements of $G^\ast$.
Furthermore, there is a bijection $\chi\mapsto\psi$ from
$\mathcal{E}(G,(s))$ to $\mathcal{E}(\bC_{G^\ast}(s),(1))$ such that
\begin{equation}\label{lusztig}
\chi(1)=\frac{|G|_{p'}}{|\bC_{G^\ast}(s)|_{p'}}\psi(1).
\end{equation}

The detailed structure of centralizers of semisimple elements in a
finite exceptional groups of Lie type was determined by Carter
\cite{Carter78}, Deriziotis \cite{Deriziotis}, and
Deriziotis-Liebeck \cite{Deriziotis-Liebeck}. A well-known result of
Steinberg states that the centralizer $\bC_{\bG^\ast} (s)$ of a
semisimple element $s$ is a connected reductive subgroup of maximal
rank in $\bG^\ast$. Such connected subgroup has a decomposition
$\bC_{\bG^\ast} (s)=\mathbf{S}\mathbf{T}$, where $\mathbf{S}$ is a
semisimple subgroup, $\mathbf{T}$ is a central torus,
$\mathbf{S}\cap \mathbf{T}$ is finite, and
$|(\bC_{G^\ast}(s))^{F^\ast}|=|\mathbf{S}^\ast||\mathbf{T}^\ast|$
(see \cite[p. 48]{Deriziotis-Liebeck}). When $s$ is in $G^\ast$, the
centralizer $\bC_{\bG^\ast} (s)$ is $F^\ast$-stable and
$\bC_{G^\ast}(s)=(\bC_{\bG^\ast} (s))^{F^\ast}$; and so
\[|\bC_{G^\ast}(s)|=|S||T|\] where $S:=\mathbf{S}^{F^\ast}$ and
$T:=\mathbf{T}^{F^\ast}$. Let $r$ be the semisimple rank of
$\bG^\ast$ and $q$ (that will be a power of $p$) the absolute value
of all eigenvalues of $F$ on the character group of an $F$-stable
maximal torus of $\bG$. Possible values for $|S|$ and $|T|$ are
available in \cite{Deriziotis,Deriziotis-Liebeck}. In particular, we
have
\[
|S|=\prod_i |L_{r_i}(q^{a_i})|
\]
and
\[
|T|=\prod_j \Phi_j(q),
\]
where $L_{r_i}(q^{a_i})$s are finite groups of Lie type (of
simply-connected type) of rank $r_i$ defined over a field of order
$q^{a_i}$ and $\Phi_j(q)$s are cyclotomic polynomials (and also
polynomials of the forms $q^2\pm \sqrt{2}q+1$, $q^2\pm \sqrt{3}q+1$,
or $q^4\pm \sqrt{2}q^3+q^2\pm \sqrt{2}q+1$ for Suzuki and Ree
groups) evaluated at $q$. As $\bC_{\bG^\ast} (s)$ has maximal rank,
we furthermore have
\begin{equation}\label{eq2}
\sum_i a_ir_i + \sum_j \deg(\Phi_j)=r.
\end{equation}

Now formula (\ref{lusztig}) implies that the typical degree of an
irreducible character of $G$ is of the form
\[
\chi(1)=\frac{|G|_{p'}}{\prod_i |L_{r_i}(q^{a_i})|_{p'} \prod_j
\Phi_j(q)} \psi(1),
\]
where $\psi\in \mathcal{E}(\bC_{G^\ast}(s),(1))$, a unipotent
character of $\bC_{G^\ast}(s)$. By \cite[Thm. 3.1]{LMT13}, for any
finite group of Lie type $\mathbf{G}^F$, where $\bG$ is a simple
algebraic group in characteristic $p$ and $F$ a Steinberg
endomorphism on $\bG$, the degree $\St(1)=|\mathbf{G}^F|_p$ of the
Steinberg character $\St$ of $\mathbf{G}^F$ is strictly larger than
the degree of any other unipotent character. Therefore, the degrees
of unipotent characters of $\bC_{G^\ast}(s)$, which are in fact the
same as those of the semisimple group $S$, are bounded by $\prod_j
|L_{r_i}(q^{a_i})|_{p}$. It follows that
\[
b(G)\leq \frac{|G|_{p'}}{\prod_i |L_{r_i}(q^{a_i})|_{p'} \prod_j
\Phi_j(q)} \prod_j |L_{r_i}(q^{a_i})|_{p},
\]
By Lemma \ref{lem:order-bound},
\[
\frac{|L_{r_i}(q^{a_i})|_{p}}{|L_{r_i}(q^{a_i})|_{p'}} \leq
\frac{1}{(q^{a_i}-1)^{r_i}}\leq \frac{1}{(q-1)^{a_ir_i}}.
\]
Also, it is easy to see that \[ \Phi_j(q)\geq (q-1)^{\deg \Phi_j}.
\]
We therefore deduce that
\[
b(G)\leq \frac{|G|_{p'}}{(q-1)^{\sum_i a_ir_i + \sum_j
\deg(\Phi_j)}}= \frac{|G|_{p'}}{(q-1)^r}.
\]
On the other hand, we have $|G|_{p'}\leq |G|_pq^r$ by again Lemma
\ref{lem:order-bound}, and it follows that
\[
\frac{b(G)}{|G|_p}\leq \frac{q^r}{(q-1)^r}.
\]
As the rank $r$ is at most $8$ for exceptional groups, the desired
inequalities follow.
\end{proof}

Bounds for $b(S)$ of alternating groups and classical groups are
collected in the following.

\begin{lem}\label{lem:bounds-for-b(S)}
Let $S$ be a finite simple group, $n$ a positive integer, and $q$ a
prime power.

\begin{enumerate}[\rm(i)]

\item For $S=\Al_n$ with $n\geq 5$, \[\frac{1}{2}e^{-1.28255\sqrt{n}}\sqrt{n!} \leq b(S)\leq
e^{-0.11565\sqrt{n}}\sqrt{n!}.\] In particular,
\[
b(S)>\frac{1}{2}e^{-1.28255\sqrt{n}} (2\pi n)^{1/4}(n/e)^{n/2}
\]

\item For $S=\Al_n$ with $n\geq 5$, \[b(\Al_{n+1})\geq \frac{2(n+1)}{\sqrt{8n+1}+3}b(\Al_n).\]

\item For $S=PSL_n(q)$ with $n\geq 2$,

\[{b(S)}<13(1+\log_q(n+1))^{2.54}\St_S(1).\]

\item For $S=PSU_n(q)$ with $n\geq 3$,
\[ {b(S)}<2(2+\log_q(n+1)^{1.27}\St_S(1).\]

\item For $S\in\{\Omega_{2n+1}(q), PSp_{2n}(q),P\Omega_{2n}^\pm(q)\}$ with $n\geq 2$ and $q$ odd,
\[ {b(S)}<38(1+\log_q(2n+1))^{1.27}\St_S(1).\]

\item For $S\in \{\Omega_{2n+1}(q), PSp_{2n}(q),P\Omega_{2n}^\pm(q)\}$ with $n\geq 2$ and $q$ even,
\[ {b(S)}<8(1+\log_q(2n+1))^{1.27}\St_S(1).\]

\end{enumerate}
\end{lem}

\begin{proof}
Part (i) follows from \cite[Thm. 1]{VK85} and Parts (iii)-(vi)
follow from \cite[Thm. 5.1, 5.2, and 5.3]{LMT13}. Let us prove Part
(ii).

Let $\chi\in\irr(\Sy_{n})$ such that $\chi(1)=b(\Sy_{n})$. Let
$\lambda$ be the partition of $n$ corresponding to $\chi$ and
$Y_\lambda$ be the Young diagram associated to $\lambda$. By the
well-known branching rule, the induction $\chi^{\Sy_{n+1}}$ of
$\chi$ from $\Sy_n$ to $\Sy_{n+1}$ is the sum of irreducible
characters corresponding to the partitions of $n+1$ whose associated
Young diagrams are obtained from $Y_\lambda$ by adding a suitable
node. The number of those suitable nodes is at most
$(\sqrt{8n+1}+1)/2$ (see \cite[p. 1950]{HHN16}) and at most one of
the resulting Young diagrams is symmetric. We deduce that
\[
(n+1)b(\Al_n)\leq (n+1)\chi(1)=\chi^{\Sy_{n+1}}(1)\leq
\frac{\sqrt{8n+1}+3}{2} b(\Al_{n+1}),
\]
and the result follows.
\end{proof}


\section{Theorem~\ref{thm:main4}: Groups of Lie
type}\label{sec:theoremD-Lie}

In this section we prove Theorem \ref{thm:main4} when the groups
involved are of Lie type.

In the following, for simplicity, we say that two groups have the
same defining characteristic if there is a common characteristic
over which the groups can be defined.

\begin{prop}\label{prop:same-char}
Let $S$ and $H$ be finite simple groups of Lie type. Suppose that
$\cod(S)\subseteq \cod(H)$. Then $S$ and $H$ have the same defining
characteristic.
\end{prop}

\begin{proof} Suppose that the defining characteristic of $H$ is
$p$. By Lemma~\ref{lem-chi-p-power}, $|S|_{p'}=|H|_{p'}$ and there
is $\chi\in\irr(S)$ such that $\chi(1)=|S|_p$. By
Lemma~\ref{lem-x=H/S}, it follows that

\begin{equation}\label{eq1}
d\cdot\frac{|H|_p}{\chi(1)}\in \cd(H) \text{ for every } d\in\cd(S).
\end{equation}

Certainly if $\chi$ is the Steinberg character of $S$ then we are
done. So we assume otherwise and aim to find a contradiction or end
up with a case where $H$ can be defined in another characteristic
not equal to $p$. By the classification of prime-power-degree
representations of quasi-simple groups \cite[Thm.
1.1]{Malle-Zalesskii}, we arrive at the following possibilities of
$S$ and $\chi(1)$.

\medskip

(i) $S=PSL_2(q)$, $\chi(1)\in\{q\pm 1\}$ or $q$ is odd and
$\chi(1)\in \{(q\pm 1)/2\}$. Observe that $\chi(1)$ cannot be $(q\pm
1)/2$ because otherwise, by taking $d=2\chi(1)$, we would have
$2|H|_p\in\cd(H)$, which is impossible. So $\chi(1)=q+\alpha=p^x$
for some $\alpha\in\{\pm 1\}$ and $x\in \mathbb{N}$. Suppose first
that $q=2^t$ for some $t\geq 2$. Then $2^t+\alpha=p^x$. By
Mihailescu's theorem \cite{Mih04} (previously known as Catalan's
conjecture), either $x=1$ so that $2^t+\alpha$ is a (Mersenne or
Fermat) prime or $\alpha=1$ and $t=3$. In the latter case, $p=3$ and
$|H|_{3'}=|S|_{3'}=56$, forcing $H$ to be ${}^2G_2(3)'$, which turns
out to be isomorphic to $S=PSL_2(8)$, as desired. So it remains to
consider the former case: $q+\alpha=2^t+\alpha=p$ is the defining
characteristic of $H$. Now $|H|_{p'}=|S|_{p'}=q(q-\alpha)$. Let
$p^a$ be the order of the underlying field of $H$. It is clear from
the order formulas of simple groups of Lie type (see \cite[p.
xvi]{Conway}) that $|H|_p< |H|_{p'}/(p^a-1)\leq |H|_{p'}/(p-1)$. We
therefore deduce that
\[|H|_p<\frac{q(q-\alpha)}{q+\alpha-1}.\]
Thus we must have $\alpha=-1$ and $|H|_p=p$. Now $H$ is a simple
group of Lie type in characteristic $p$ such that $|H|=p(p+1)(p+2)$.
This is impossible as for such a group $H$, one can check from the
order formula that $|H|_{p'}<(|H|_p)^2$.

Next we suppose $q\geq 5$ is odd. Then $p=2$ and $q+\alpha =|S|_2$.
Again by Mihailescu's theorem, either $q$ is a prime or $\alpha=-1$
and $q=9$. The case $q=9$ is eliminated in the same way as before.
So assume that $|H|_{2'}=q(q-\alpha)/2$ and $q$ is a prime. Note
that when $H$ is not of type $A_1$, every prime divisor of
$|H|_{p'}$ is smaller than $\sqrt{|H|_{p'}}$. Therefore our group
$H$ must be $PSL_2(q_1)$ for some $2$-power $q_1$, implying that
$q(q-\alpha)/2=q_1^2-1$. This, however, returns no relevant
solutions.

\medskip

(ii) For the remaining possibilities of $S$ and $\chi$, the
character $\chi$ has a decent small degree and we are able to
produce an irreducible character of $S$ whose degree is a proper
multiple of $\chi(1)$. Condition (\ref{eq1}) then implies that a
proper multiple of $|H|_p$ is a character degree of $H$, which is
impossible. This required character turns out to be chosen as a
unipotent character in most cases. We refer the reader to \cite[\S
13.8]{Carter85} for the description of unipotent characters of
finite classical groups.

The next possibility of $S$ and $\chi(1)$ is $S=PSL_n(q)$, $q>2$,
$n$ is an odd prime, $(n,q-1)=1$, and $\chi(1)=(q^n-1)/(q-1)$. If
$n=3$ then $SL_3(q)$ has an irreducible character of degree $q^3-1$
(see \cite{Simpson-Frame73}), which is a proper multiple of
$\chi(1)=q^2+q+1$. For $n\geq 5$, the unipotent character
parameterized by the partition $(2,n-2)$ with degree
\[
\chi^{(2,n-2)}(1)=\frac{(q^n-1)(q^{n-1}-q^2)}{(q-1)(q^2-1)}
\]
fulfills our requirement.

Another possibility is $S=PSU_n(q)$, $n$ is an odd prime,
$(n,q+1)=1$, and $\chi(1)=(q^n+1)/(q+1)$. This case is handled
similarly as in the previous one. Here we note that, when $n\geq 5$
is odd, the unipotent character parameterized by the partition
$(2,n-2)$ has degree
\[
\chi^{(2,n-2)}(1)=\frac{(q^n+1)(q^{n-1}-q^2)}{(q+1)(q^2-1)}.
\]

The next case is $S=PSp_{2n}(q)$, $n\geq 2$, $q=\ell^t$ with $\ell$
an odd prime, $tn$ is a $2$-power, and  $\chi(1)=(q^n+1)/2$. Now the
unipotent character parameterized by the symbol ${0\hspace{6pt} 1
\choose n}$ has required degree
\[
\chi^{{0\hspace{6pt} 1 \choose
n}}(1)=\frac{(q^n+1)(q^{n}+q)}{2(q+1)}.
\]

The last possibility involving a family of groups is
$S=PSp_{2n}(3)$, $n\geq 3$ is a prime, and $\chi(1)=(3^n-1)/2$. Then
$S$ has a unipotent character with degree
\[
\chi^{{0\hspace{6pt} 1 \hspace{6pt} n\choose
-}}(1)=\frac{(3^n-1)(3^{n}-3)}{8}.
\]

\medskip

(iii) $(S,\chi(1))\in\{(Sp_6(2),7)$, $(Sp_6(2), 27)$, $({}^2F_4(2)',
27)$, $(G_2(2)', 7)$, $(G_2(2)', 27)$, $(G_2(3), 64)\}$. First
assume that $S=G_2(2)'$ and $\chi(1)=27$. Then $p=3$ and
$|H|_{3'}=|S|_{3'}=224$. It is easy to check that there is no such
simple group $S$ of Lie type in characteristic 3 with $27\mid
|S|_3$. In all other cases one can find a character $\psi\in\irr(S)$
such that $\chi(1) \mid \psi(1)$ but $\chi(1)<\psi(1)$. Therefore
$\psi(1){|H|_p}/\chi(1)$ is a proper multiple of $|H|_p$ and thus
cannot be a character degree of $H$, violating condition
(\ref{eq1}).
\end{proof}


As seen in Lemma \ref{lem-chi-p-power} and Proposition
\ref{prop:same-char}, we face the situation where two simple groups
$S$ and $H$ of Lie type have the same characteristics $p$ and
$|S|_{p'}=|H|_{p'}$. It is surprising to us that this turns out to
happen only when $|S|=|H|$ (see Proposition \ref{prop:|S|_p'=|H|_p'}
below), and therefore they are among the coincidences appeared in
the simple order theorem.

We slightly modify two of the invariants in Artin's proof of the
simple order theorem (for classical groups) to prove our results.

\begin{define}
Let $S$ be a finite group of Lie type in characteristic $p$. Let
$\omega=\omega(S)$ and $\psi=\psi(S)$ respectively denote the
largest and the second largest orders of $p$ modulo a prime divisor
of $|S|_{p'}$. We will refer to $\omega(S)$ and $\psi(S)$ as the
Artin invariants of $S$.
\end{define}

In fact, when $p$ is the dominant prime of $|S|$, these $\omega(S)$
and $\psi(S)$ coincide with Artin's invariants defined in
\cite{Artin55}. We remark that there are only a few cases involving
Mersenne and Fermat primes where $p$ is not dominant in $|S|$, and
they are listed in \cite[Thm. 3.3]{Kimmerle-et-al}.

Assume for now that $S$ is not one of $G_2(2)'$, ${}^2G_2(3)'$, and
${}^2F_4(2)'$. (Note that $S_1=G_2(2)'\cong PSU_3(9)$ and
$S_2={}^2G_2(3)'\cong PSL_2(8)$ and, even though we do not allow
$S_1$ (or $S_2$) to be viewed as a Lie type group over a field of
order 2 (or 3), we do allow it to be viewed as one over the field of
order 9 (or 8).) Let $q=p^t$ be the order of the underlying field of
$S$. It is well-known that the order $|S|$ then has the standard
cyclotomic factorization in terms of $q$ as
\[
|S|=\frac{1}{d}q^k \prod_i\Phi_i(q)^{e_i},
\]
where $d$, $k$, $e_i$s depend on $S$ and can be found in \cite[Table
6]{Conway} and \cite[Tables C1, C2, and C3]{Kimmerle-et-al} for
instance, and $\Phi_i(q)$ is the value of the $i$th cyclotomic
polynomial evaluated at $q$. Replacing $q$ by $p^n$ and factorizing
each $\Phi_i(x^t)$ further into cyclotomic polynomials of $x$, one
has
\[
|S|=\frac{1}{d}p^{kt} \prod_i\Phi_i(p)^{f_i}
\]
for certain positive integers $f_i$s depending on $S$.

Using Zsigmondy's theorem, it is not difficult to see that, except
for some `small' cases, the invariants $\omega(S)$ and $\psi(S)$ are
precisely the largest and second largest, respectively, index $i$
such that $\Phi_i(p)$ appears in the cyclotomic factorization of
$|S|$ (see \cite[Lem. 4.6]{Kimmerle-et-al}). We refer the reader to
\cite[Tables A1, A2 and A3]{Kimmerle-et-al} for the list of
exceptions and the values of their Artin's invariants, including the
groups $G_2(2)'$, ${}^2G_2(3)'$, and ${}^2F_4(2)'$ we excluded
earlier. We reproduce in Table \ref{Artin invariants} the values of
$\omega(S)$ and $\psi(S)$ for the generic case only.

\begin{table}[ht]
\caption{$\omega(S)$ and $\psi(S)$ for simple groups of Lie type:
generic case.\label{Artin invariants}}
\begin{center}
\begin{tabular}{llll}
\hline \begin{tabular}{l} $S$\\ ($q=p^t$)\end{tabular} &
\begin{tabular}{l} Conditions\\ ($p$ a Mersenne prime)\end{tabular}&
$\omega(S)$& $\psi(S)$ \\ \hline

$PSL_n(q), n\geq 2$& \begin{tabular}{l} $(n,q)\neq (2,2^6),(3,2^2),(3,2^3)$,\\
$(4,2^2), (6,2),(7,2),(2,p^2),(3,p)$\end{tabular}  & $nt$ & $(n-1)t$ \\

$PSU_4(q)$&   & $6t$ & $4t$ \\

$PSU_n(q), n\geq 3$ odd& $(n,q)\neq (3,2^3),(5,2),(3,p)$  & $2nt$ & $2(n-2)t$ \\

$PSU_n(q), n\geq 6$ even& $(n,q)\neq (6,2)$  & $2(n+1)t$ & $2(n-1)t$ \\

$\Omega_{2n+1}(q), n\geq 2$& $(n,q)\neq (2,2^8),(3,2),(4,2),(2,p)$  & $2nt$ & $2(n-1)t$ \\

$PSp_{2n}(q), n\geq 3$&  & $2nt$ & $2(n-1)t$ \\

$P\Omega^+_{2n}(q), n\geq 4$& $(n,q)\neq (4,2),(5,2)$  & $2(n-1)t$ & $2(n-2)t$ \\

$P\Omega^-_{2n}(q), n\geq 4$& $(n,q)\neq (4,2)$  & $2nt$ & $2(n-1)t$ \\

${}^2B_2(2^t), t\geq 3$ odd& $t\equiv 3(\bmod 6)$  & $4t$ & $4t/3$ \\

${}^2B_2(2^t), t\geq 3$ odd& $t\equiv \pm1(\bmod 6)$  & $4t$ & $t$ \\

$G_2(q), q\geq 3$& $q\neq 4$  & $6t$ & $3t$ \\

${}^2G_2(3^t), t\geq 3$ odd&  & $6t$ & $2t$ \\

${}^3D_4(q)$& $q\neq 2$ & $12t$ & $6t$ \\

$F_4(q)$&  & $12t$ & $8t$ \\

${}^2F_4(2^t), t\geq 3$ odd&  & $12t$ & $6t$ \\

$E_6(q)$&  & $12t$ & $9t$ \\

${}^2E_6(q)$&  & $18t$ & $12t$ \\

$E_7(q)$&  & $18t$ & $14t$ \\

$E_8(q)$&  & $30t$ & $24t$ \\

\hline
\end{tabular}
\end{center}
\end{table}

\begin{prop}\label{prop:|S|_p'=|H|_p'}
Let $p$ be a prime. Suppose that $S$ and $H$ are two non-isomorphic
simple groups of Lie type in characteristic $p$ and
$|S|_{p'}=|H|_{p'}$. Then $\{S,H\}=\{PSL_4(2),PSL_3(4)\}$ or
$\{S,H\}=\{\Omega_{2n+1}(q), PSp_{2n}(q)\}$ for some $n\geq 3$ and
odd $q$. In particular, $|S|=|H|$.
\end{prop}

\begin{proof}
By the assumptions, we have $\omega(S)=\omega(H)$ and
$\psi(S)=\psi(H)$. First we consider the case where both $S$ and $H$
are generic so that their invariants $\omega$ and $\psi$ are
available in Table \ref{Artin invariants}. Comparing those values,
we can find all the collections of groups with equal values of
$\omega$ and $\psi$. We list these collections in Table \ref{Artin
invariants2} (each row in the table is one such collection). Now one
simply compare the $p'$-parts of orders of groups in each
collection. It turns out that the only pair with the same $p'$-parts
of orders is $\{\Omega_{2n+1}(q), PSp_{2n}(q)\}$ with some $n\geq 3$
and odd $q$.

Assume now that at least one of the two groups, say $S$, is
non-generic. That is, $S$ is among the exceptions listed in the
second column of Table \ref{Artin invariants}. The values of the
invariants $\omega$ and $\psi$ of these groups are available in
\cite[Tables A2 and A3]{Kimmerle-et-al}. The analysis is basically
the same as in the generic case, but more tedious. We first find all
the possible groups $H$ with $\omega(S)=\omega(H)$ and
$\psi(S)=\psi(H)$, and then compare $|S|_{p'}$ and $|H|_{p'}$, where
$p$ is the defining characteristic of $S$ and $H$.

Let us demonstrate the case $S=PSL_3(4)$ as an example. Then
$\omega(S)=4$ and $\psi(S)=3$. But there are only two other simple
groups of Lie type with the same values of $\omega$ and $\psi$,
namely $PSL_4(2)$ and $P\Omega^+_8(2)$. However,
$|P\Omega^+_8(2)|_{2'}\neq |PSL_3(4)|_{2'}=|PSL_4(2)|_{2'}$, and so
we come up with another possible pair for $\{S,H\}$, namely
$\{PSL_4(2),PSL_3(4)\}$, as stated in the theorem.
\end{proof}

\begin{table}[ht]
\caption{Simple groups of Lie type with the same values of $\omega$
and $\psi$: generic case.\label{Artin invariants2}}
\begin{center}
\begin{tabular}{l}
\hline

$PSL_n(p^{2s})$, $\Omega_{2n+1}(p^s)$, $PSp_{2n}(p^s)$, $P\Omega^+_{2(n+1)}(p^s)$, $P\Omega^-_{2n}(p^s)$ \\

$PSL_3(p^{2s})$, $PSU_4(p^s)$, $\Omega_{7}(p^s)$,
$PSp_{6}(p^s)$, $P\Omega^+_{8}(p^s)$\\

$PSL_2(p^{6s})$, $\Omega_5(p^{3s})$, $G_2(p^{2s})$, ${}^3D_4(p^s)$\\

$PSL_2(p^{3s})$, $G_2(p^s)$\\

$PSL_3(p^{4s})$, $PSU_4(p^{2s})$, $\Omega_{7}(p^{2s})$,
$PSp_{6}(p^{2s})$, $P\Omega^+_{8}(p^{2s})$,   $F_4(p^s)$\\

$PSL_2(2^{6s})$, $\Omega_5(2^{3s})$, $G_2(2^{2s})$, ${}^3D_4(2^s)$, ${}^2F_4(2^s)$, $s\geq 3$ odd\\

$PSL_4(p^{3s})$, $E_6(p^s)$\\

$PSL_4(p^{6s})$, $\Omega_9(p^{3s})$, $P\Omega_{10}^+(p^{3s})$,
$P\Omega_{8}^-(p^{3s})$, $E_6(p^{2s})$\\

$PSL_3(p^{6s})$, $PSU_4(p^{3s})$, $\Omega_{7}(p^{3s})$,
$PSp_{6}(p^{3s})$, $P\Omega^+_{8}(p^{3s})$, ${}^2E_6(p^s)$\\

$PSL_3(p^{12s})$, $\Omega_{7}(p^{6s})$, $PSp_{6}(p^{6s})$, $P\Omega^+_{8}(p^{6s})$, $F_4(p^{3s})$, ${}^2E_6(p^{2s})$\\

$PSL_5(p^{6s})$, $\Omega_{11}(p^{3s})$, $PSp_{10}(p^{3s})$, $P\Omega_{12}^+(p^{3s})$, $P\Omega_{10}^-(p^{3s})$, $E_8(p^s)$\\

$PSU_n(q)$, $PSU_{n-1}(q)$, $n\geq 7$ odd\\

$PSU_3(2^{2s})$, ${}^2B_2(2^{3s})$, $s$ odd\\

$PSU_3(3^{s})$, ${}^2G_2(3^{s})$, $s\geq 3$ odd\\

$PSU_9(p^s)$, $PSU_8(p^s)$, $E_7(p^s)$\\

$\Omega_{2n+1}(q)$, $PSp_{2n}(q)$, $n\geq 3$, $q$ odd\\

\hline
\end{tabular}
\end{center}
\end{table}

The next theorem improves Theorem \ref{cor-cod(S)=cod(H)} when the
relevant groups are of Lie type.

\begin{thm}\label{thm:Lie type}
Let $S$ and $H$ be finite simple groups of Lie type such that
$\cod(S)\subseteq \cod(H)$. Then $S\cong H$.
\end{thm}

\begin{proof}
By Lemma \ref{lem-chi-p-power} and Propositions \ref{prop:same-char}
and \ref{prop:|S|_p'=|H|_p'}, we have that $S$ and $H$ fall into one
of two pairs of groups concluded in Proposition
\ref{prop:|S|_p'=|H|_p'}. The result now follows by Lemma
\ref{lem:O2n+1andPSp}.
\end{proof}


\section{Theorem~\ref{thm:main4}: The mixed case of alternating groups and groups of Lie
type}\label{sec:theoremD-mixed-case}

In this section, we prove Theorem~\ref{thm:main4} in the mixed
situation where the set of codegrees of an alternating group $S$ is
contained in that of a simple group $H$ of Lie type, or vice versa.


In the following proposition we remark that the condition on $m$  is
necessary, due to the coincidences of isomorphic simple groups:
$A_5\cong PSL_2(4)\cong PSL_2(5)$, $A_6\cong PSL_2(9)$, and
$A_8\cong PSL_4(2)$. We also recall that $f(X):=|X|/b(X)$, where
$b(X)$ is the largest character degree of $X$.

\begin{prop}\label{prop:mixed case I}
Suppose $m=7$ or $m\geq 9$. Let $H$ be a simple group of Lie type.
If $|\Al_m|$ divides $|H|$, then $f(H)>f(\Al_m)$. As a consequence,
$\cod(\Al_m)\nsubseteq \cod(H)$.
\end{prop}

\begin{proof}
Let $p$ be the defining characteristic and $q$, a power of $p$, the
order of the underlying field of $H$. Consider first the case $H$
being of exceptional type. Using Lemma \ref{lem:exceptional-b(S)},
we have
\[
f(H)> \frac{|H|}{256|H|_p}=\frac{1}{256}|H|_{p'}.
\]
As $|\Al_m|$ divides $|H|$, it follows that $ f(H)>
(1/256)|\Al_m|_{p'}.$ Therefore, to prove the theorem, it suffices
to show \[b(\Al_m)\geq 256|\Al_m|_p.\] Let us assume for now that
that $m\geq 10$. In particular, the dominant prime in $|\Al_m|=m!/2$
is $2$. We therefore just need to show $b(A_m)\geq 256|\Al_m|_2$. As
$|\Al_m|_2\leq 2^{m-2}$, for this we want to show
\[
b(\Al_m)\geq 64\cdot2^m.
\]
Note that $b(\Al_{19})=64,664,600>64\cdot 2^{19}$ (see \cite{McK86}
for the degree of the largest irreducible characters and associated
partitions of symmetric groups of degree up to 75, from which one
can deduce the exact value or a good bound for the one of
corresponding alternating groups). Now one just inducts on $m$ with
the help of Lemma \ref{lem:bounds-for-b(S)}(ii) to achieve the
desired bound for $m\geq 19$.

Suppose that $m\leq 18$ and recall that we are still dealing with
exceptional groups. When $q=2$, the proposition can be verified
directly, so assume that $q\geq 3$. In such case, $b(H)<26|H|_p$ by
Lemma \ref{lem:exceptional-b(S)}, and whence the above estimate can
be refined so that we only need to prove $b(\Al_m)\geq 26|\Al_m|_p$,
which turns out to be true for all $18\geq n\geq 13$. For the
remaining values $m\leq 12$, the arguments go as follows. First we
are done if $f(\Al_m)\leq \sqrt{|H|}$, as $f(H)>\sqrt{|H|}$, so we
may assume that $|H|<f(\Al_m)^2$. For each $m\leq 12$, there are
indeed no possibilities for $H$ satisfying $|H|<f(\Al_m)^2$ and
$|A_m|$ divides $|H|$.


Following the same idea as in the case of exceptional groups, but
using Lemma~\ref{lem:bounds-for-b(S)} instead, we can show that in
fact $f(H)>f(\Al_m)$ for every $H$ of classical type and $m\geq 19$.
Let us present the details for only the linear groups.

Consider $H=PSL_n(q)$ for some $n\geq 2$ and $q$ a prime power. By
Lemma \ref{lem:bounds-for-b(S)}(i), we have
\[f(\Al_m)=\frac{m!}{2b(\Al_m)}\leq e^{1.28255\sqrt{m}}(m!)^{1/2}.\]
Thus, if $|H|\geq e^{2.5651\sqrt{m}}m!$ then $f(H)>\sqrt{|H|}\geq
e^{1.28255\sqrt{m}}(m!)^{1/2}\geq f(\Al_m)$ and we would be done. We
therefore can assume that $|H|< e^{2.5651\sqrt{m}}m!$, which in
particular implies that $n<m$. Using Lemma
\ref{lem:bounds-for-b(S)}(iii), we see that, as before, it is enough
to show that $b(\Al_m)\geq 13(1+\log_q(n+1))^{2.54}|\Al_m|_p$. Since
$m\geq n+1$ and $|\Al_m|_p\leq 2^{m-2}$, for this it is sufficient
to show that
\[
b(\Al_m)\geq 13(1+\log_2m)^{2.54}2^{m-2}.
\]
This last inequality is indeed true for $m=20$, and therefore is
true for all $m\geq 20$, by induction and Lemma
\ref{lem:bounds-for-b(S)}(ii). Checking directly, we see that the
inequality $b(\Al_m)\geq 13(1+\log_q(n+1))^{2.54}|\Al_m|_2$ is still
valid for $n=19$.

As for the exceptional types, we are left to consider the small
cases $m\leq 18$. Again we are done if $f(\Al_m)\leq \sqrt{|H|}$, so
we may assume that $|H|<f(\Al_m)^2$. For each $m$, we search for
relevant $H$ satisfying $|H|<f(\Al_m)^2$ and $|A_m|$ divides $|H|$
and find that, for such an $H$, the inequality $f(H)>f(\Al_m)$
always holds true.
\end{proof}


We shall need the following result on $2$-defect zero and $3$-defect
zero characters of alternating groups, which easily follows from
earlier work of F. Garvan, D. Kim and D. Stanton \cite{GKS90} on the
so-called \emph{$p$-core partitions}. They are  partitions having no
hook lengths divisible by $p$. Using Garvan-Kim-Stanton's result, A.
Granville and K. Ono \cite{Granville-Ono} proved the existence of
$p$-defect zero characters with $p\geq 5$ in symmetric and
alternating groups.

\begin{lem}\label{lem:defect zero Am}
Let $m$ be a positive integer.
\begin{enumerate}[\rm(i)]
\item $\Al_m$ has a $2$-defect zero irreducible character if and only
if $m=2k^2+k$ or $m=2k^2+k+2$ for some $k\in\N$.

\item $\Al_m$ has a $3$-defect zero irreducible character if and only
if there is a prime $\ell\equiv 2(\bmod 3)$ such that the the exact
power of $\ell$ dividing $3m+ 1$ is odd.
\end{enumerate}
\end{lem}

\begin{proof}
See the discussion in \cite[pp. 333-334]{Granville-Ono}.
\end{proof}

\begin{prop}\label{prop:mixed case II}
Let $S$ be a simple group of Lie type and $8\neq m\geq 7$ an
integer. Then $\cod(S)\nsubseteq \cod(\Al_m)$. In fact, if $|S|$
divides $|\Al_m|$ and $m\geq 44$, then $f(S)<f(\Al_m)$.
\end{prop}

\begin{proof}
Assume by contradiction that $\cod(S)\subseteq \cod(\Al_m)$. By
Lemma~\ref{lem:fS geq fH}, we then have $f(S)\geq f(\Al_m)$.

Suppose that the defining characteristic of $S$ is $p$. Observe that
$f(S)\leq |S|/\St_S(1)=|S|_{p'}$. Furthermore, $|S|_{p'}<|S|_p^2$
(see \cite[Proof of Thm. 12]{Cossey}) and $|S|_p\leq |\Al_m|_p$ by
Theorem~\ref{thm-|S|divides|G|}. Therefore we have $f(S)<
(|\Al_m|_p)^2$. Assume for a moment that $m\geq 10$ so that
$|\Al_m|_p\leq |\Al_m|_2\leq 2^{m-2}$. We now have $f(S)< 2^{2m-4}$.
On the other hand, it is clear that $f(\Al_m)>\sqrt{m!/2}$.
Therefore, we would be done if $m!\geq 2^{4m-7}$. By the well-known
estimate $m!>\sqrt{2\pi m}(m/e)^m$, this is certainly true when
$m\geq44$. So we may now suppose that $m\leq 43$.

As mentioned above, every simple group of Lie type, and therefore
$S$ in particular, has a $2$-defect zero irreducible character,
which means that $S$ has an odd codegree and so does $\Al_m$ as
$\cod(S)\subseteq \cod(\Al_m)$. It follows that $m=2k^2+k$ or
$m=2k^2+k+2$ for some $k\in\N$, by Lemma~\ref{lem:defect zero
Am}(i). This forces $m$ to be one of $10, 12, 21, 23, 36$, or $38$.
By the same reason, $\Al_m$ has a codegree not divisible by $3$ and
so Lemma~\ref{lem:defect zero Am}(ii) further narrows down the
choices for $m$: $m\in\{10,12,21,36\}$. In fact, when $m=21$ or
$36$, we still have $f(\Al_m)> |\Al_m|_2^2$, and since
$|\Al_m|_2^2>f(S)$, it follows that $f(\Al_m)>f(S)$, which is a
contradiction.

Suppose $m=10$. The inequality $f(\Al_m)<(|\Al_m|_p)^2$ then forces
$p=2$ or $3$. If $p=2$ then $|S|_{2'}=|\Al_{10}|/\chi(1)$, where
$\chi\in\irr(\Al_{10})$ is one of the two $2$-defect zero
irreducible characters of equal degree $384$, implying
$|S|_{2'}=10!/(2\cdot 384)=4725$. It is easy to see from
\cite{Conway} that there is no such group of Lie type in
characteristic $2$. If $p=3$ then $|S|_{3'}=10!/(2\cdot 567)=3200$
since $\Al_{10}$ has a unique $3$-defect zero character of degree
$567$, which again leads to a contradiction as there is no such
group in characteristic $3$. The case $m=12$ is treated similarly
and we skip the details.
\end{proof}


\section{Theorem~\ref{thm:main4}: Alternating and sporadic
groups}\label{sec:theoremD-alternating-sporadic}

\begin{prop}\label{prop:alternating}
Let $m<n$ be positive integers. Then $f(\Al_m)<f(\Al_n)$.
Consequently, $\cod(\Al_m)\nsubseteq \cod(\Al_{m+1})$.
\end{prop}

\begin{proof}
It suffices to show that $b(\Al_{m+1})<(m+1)b(\Al_m)$. Let
$\chi\in\irr(\Al_{m+1})$ such that $\chi(1)=b(\Al_{m+1})$. As shown
in \cite[p. 1956]{HHN16}, such $\chi$ must be the restriction of an
irreducible character, say $\psi$, of $\Sy_{m+1}$ whose associated
partition, say $\lambda$, is not self-conjugate. In particular,
$\chi(1)=\psi(1)$. As in the proof on Lemma
\ref{lem:bounds-for-b(S)}(ii), let $Y_\lambda$ be the Young diagram
associated to $\lambda$. The restriction $\psi_{\Sy_m}$ of $\psi$ to
$\Sy_n$ is the sum of irreducible characters corresponding to the
partitions of $n$ whose associated Young diagrams are obtained from
$Y_\lambda$ by removing a suitable node. The number of those
suitable nodes is at most $(\sqrt{8m+9}-1)/2$, so
\[
b(\Al_{m+1})=\psi(1)\leq \frac{\sqrt{8m+9}-1}{2} b(\Sy_m).
\]
Since $b(\Sy_m)<2b(\Al_m)$ as already mentioned above, it follows
that
\[
b(\Al_{m+1})< (\sqrt{8m+9}-1) b(\Al_m),
\]
which implies our desired inequality $b(\Al_{m+1})<(m+1)b(\Al_m)$
for $m\geq 5$. The result is easily checked for smaller $m$.
\end{proof}

\begin{prop}\label{prop:sporadic}
Theorem \ref{thm:main4} is true when either $S$ or $H$ is a sporadic
simple group.
\end{prop}

\begin{proof}
The case where both $S$ and $H$ are sporadic simple groups can be
verified by using the available data in \cite{Conway}.

Suppose that $S$ is a sporadic group and $H=\Al_m$ for some $m\geq
5$. Let $p_S$ be the largest prime divisor of $|S|$. By Theorem
\ref{thm-|S|divides|G|}, we have $|S|$ divides $|\Al_m|$, so
$p_S\leq m$. By Lemma \ref{lem:fS geq fH}, we have $f(S)\geq
f(\Al_m)>\sqrt{m!/2}$. It follows that $f(S)\geq \sqrt{p_S!/2}$.
Again using \cite{Conway}, it can be checked that this can never
happen.

Next we assume that $S$ is a sporadic group and $H$ is a simple
group of Lie type in characteristic $p$. Suppose first that $S$ has
an irreducible character, say $\chi$, of $p$-defect zero. Then, as
argued in the proof of Lemma \ref{lem-chi-p-power}, we have
$|S|_{p'}=|H|_{p'}$ and $\chi(1)=|S|_{p}$. In particular, $\chi(1)$
is a prime power, and therefore, \cite[Thm. 1.1]{Malle-Zalesskii}
yields
\[
(S,p,\chi(1))\in\{(M_{11}/M_{12}, 11,11), (M_{11},2,16),
(M_{24}/Co_{2}/Co_{3},23,23)\}
\]
(We note that $M_{12}$ has another irreducible character of prime
power degree, namely $16$, but the character is not of $2$-defect
zero and thus does not fit our situation.) However, for each of
these possibilities, there is no simple group of Lie type $H$ in
characteristic $p$ such that $|H|_{p'}=|S|_{p'}$. Next, we suppose
that $S$ has no characters of $p$-defect zero. By
\cite[Cor.~2]{Granville-Ono}, $$p\in\{2,3\} \text{ and }
S\in\{M_{12}, M_{22}, M_{24}, J_2, HS, Suz, Ru, Co_1, Co_3,BM\}.$$
Now we just apply Lemma~\ref{lem:bounds-for-b(S)} and argue
similarly as in the proof of Proposition~\ref{prop:mixed case I},
with $S$ in place of $\Al_m$, to arrive at $f(H)>f(S)$, and thus it
follows from Lemma \ref{lem:fS geq fH} that
$\cod(S)\nsubseteq\cod(H)$.

Now we consider the case where $S=\Al_m$ for some $m\geq 5$ and $H$
a sporadic simple group. Using Theorem \ref{thm-|S|divides|G|}, we
have $m!/2$ divides $|H|$ and so $m$ is at most $\overline{p}_H-1$,
where $\overline{p}_H$ is the smallest prime not dividing $|H|$.
This constraint is enough to ensure that $f(A_m)<f(H)$, and thus
$\cod(\Al_m)\nsubseteq \cod(H)$ by Lemma \ref{lem:fS geq fH}.

Finally we consider the case where $S$ is a simple group of Lie type
and $H$ a sporadic simple group. As in the proof of
Proposition~\ref{prop:mixed case II} , we have $f(H)\leq (|H|_p)^2$,
where $p$ is the defining characteristic of $S$. The only possible
$p$ satisfying such condition is $p=2$. Now $|S|_{2'}$ is an odd
codegree of $S$, and hence of $H$, and so $|S|_{2'}=|H|/\chi(1)$ for
some $2$-defect zero character $\chi\in\irr(H)$. There are in fact
only $16$ sporadic simple groups having a $2$-defect zero
irreducible character. For such a group and such a character, there
are no $S$ satisfying the indicated condition. This concludes the
proof.
\end{proof}

Theorem \ref{thm:main4} follows from Theorem \ref{thm:Lie type} and
Propositions \ref{prop:mixed case I}, \ref{prop:mixed case II},
\ref{prop:alternating}, and \ref{prop:sporadic}.

For future work on the codegree isomorphism conjecture
(\ref{eq:HCC}), we record the following immediate consequence of
Theorem \ref{thm:main4}.

\begin{thm}\label{thm:last}
Let $S$ be a finite nonabelian simple group. Let $G$ be a minimal
counterexample to (\ref{eq:HCC}) with respect to $S$ -- that is, $G$
is minimal subject to the conditions $\cod(G)=\cod(S)$ and $G \ncong
S$. Then $G$ has a unique minimal normal subgroup $N$ and $G/N\cong
S$.
\end{thm}

\begin{proof}
Let $N$ be a maximal normal subgroup of $G$. Since
$\cod(G/N)\subseteq\cod(G)=\cod(S)$ and $G/N$ is simple, it follows
from Theorem \ref{thm:main4} that $G/N\cong S$. Furthermore, by the
minimality of $G$ as a counterexample, we have that $N$ is a minimal
normal subgroup of $G$. (If $G$ has a normal subgroup $M$ such that
$M<N$, then $\cod(G/N)\subseteq \cod(G/M)\subseteq \cod(G)$, forcing
$\cod(G/M)=\cod(S)$.) Also, $N$ is the unique minimal normal
subgroup of $G$ since, otherwise, $G=S\times S$, which violates the
assumption $\cod(G)=\cod(S)$.
\end{proof}

We conclude the paper with a couple of remarks. First, the group
pseudo-algebra $C(G)$ seems to better distinguish finite groups than
the usual complex group algebra $\mathbb{C}G$. For instance, while
any two abelian groups $A$ and $B$ of the same order have the same
complex group algebra $\mathbb{C}A=\mathbb{C}B$, it was shown in
\cite{mor23} that $A\cong B$ if and only if $C(A)=C(B)$. It has even
been speculated that a finite group $G$ and an abelian group $A$ are
isomorphic if and only if $C(G)= C(A)$. This, if true, would
indicate that abelian groups have very distinctive character
codegrees (counting multiplicities). Theorem \ref{thm:main1} shows
that simple groups indeed have very distinctive codegrees.

Our results are likely to remain true for quasi and/or almost simple
groups. However, at the time of this writing, we do not see yet a
uniform proof for these larger families of groups as the one
presented in this paper.



\begin{thebibliography}{ABCDEF}

\bibitem[Aha22]{Ahanjideh}
N. Ahanjideh, Nondivisibility among irreducible character
co-degrees, \emph{Bull. Aust. Math. Soc.} \textbf{105} (2022),
68-74.


\bibitem[Art55a]{Artin551}
E. Artin, The orders of the linear groups, \emph{Comm. Pure Appl.
Math.} \textbf{8} (1955) 355-365.

\bibitem[Art55b]{Artin55}
E. Artin, The orders of the classical simple groups, \emph{Comm.
Pure Appl. Math.} \textbf{8} (1955) 455-472.

\bibitem[BAK21]{BahriAkh}
A. Bahri, Z. Akhlaghi, and B. Khosravi, An analogue of Huppert's
conjecture for character codegrees, \textit{Bull. Aust. Math. Soc.}
\textbf{104} (2021), 278-286.

\bibitem[BBOO01]{BBOO01}
A. Balog, C. Bessenrodt, J.\,B. Olsson, and K. Ono, Prime power
degree representations of the symmetric and alternating groups,
\emph{J. London Math. Soc.} \textbf{64} (2001), 344-356.


\bibitem[BNOT15]{BNOT15}
C. Bessenrodt, H.\,N. Nguyen, J.\,B. Olsson, H.\,P. Tong-Viet,
Complex group algebras of the double covers of the symmetric and
alternating groups, \emph{Algebra Number Theory} \textbf{9} (2015),
601-628.

\bibitem[BTZ17]{Bessenrodt}
C. Bessenrodt, H.\,P. Tong-Viet, and J. Zhang, Huppert's conjecture
for alternating groups, \emph{J. Algebra} \textbf{470} (2017),
353-378.




\bibitem[Car78]{Carter78}
R.\,W. Carter, Centralizers of semisimple elements in finite groups
of Lie type, \emph{Proc. London Math. Soc.} \textbf{37} (1978),
491-507.

\bibitem[Car85]{Carter85}
R.\,W. Carter, \emph{Finite groups of Lie type. Conjugacy classes
and complex characters}, Wiley and Sons, New York et al, 1985, 544
pp.

\bibitem[CN22]{cn}
X. Chen and G. Navarro, Brauer characters, degrees and subgroups,
\textit{Bull. London Math. Soc.} \textbf{54} (2022), 891-893.



\bibitem[CH89]{Chillag1}
D. Chillag and M. Herzog, On character degrees quotients,
\textit{Arch. Math.} \textbf{55} (1989), 25-29.

\bibitem[CMM91]{Chillag-Mann}
D. Chillag, A. Mann, and O. Manz, The co-degrees of irreducible
characters, \emph{Israel J. Math.} \textbf{73} (1991), 207-223.

\bibitem[CHMN15]{Cossey}
J.\,P. Cossey, Z. Halasi, A. Mar\'{o}ti, and H.\,N. Nguyen, On a
conjecture of Gluck, \emph{Math. Z.} \textbf{279} (2015), 1067-1080.

\bibitem[Atl1]{Conway}
J.\,H. Conway, R.\,T. Curtis, S.\,P. Norton, R.\,A. Parker, and
R.\,A. Wilson, \emph{Atlas of Finite Groups}, Oxford University
Press, London 1985.


\bibitem[Der77]{Deriziotis}
D.\,I. Deriziotis, \emph{The Brauer complex and its applications to
the Chevalley groups}, Ph.D. Thesis, University of Warwick, 1977.

\bibitem[DL85]{Deriziotis-Liebeck}
D.\,I. Deriziotis and M.\,W. Liebeck, Centralizers of semisimple
elements in finite twisted groups of Lie type, \emph{J. London Math.
Soc.} \textbf{31} (1985), 48-54.


\bibitem[DM91]{DM}
F. Digne and J. Michel, {\emph{Representations of finite groups of
Lie type}}, London Mathematical Society Student Texts \textbf{21},
1991, 159 pp.

\bibitem[DL16]{dl}
N. Du and  M. Lewis, Codegrees and nilpotence class of $p$-groups,
\emph{J. Group Theory}  \textbf{19}  (2016),  561-567.




\bibitem[GKS90]{GKS90}
F. Garvan, D. Kim and D. Stanton, Cranks and $t$-cores,
\emph{Invent. Math.} \textbf{101} (1990), 1-17.


\bibitem[GKL+22]{gkl}
M. Gintz, M. Kortje, M. Laurence, Y. Liu, Z. Wang, and Y. Yang, On
the characterization of some non-abelian simple groups with few
codegrees, \emph{Comm. Algebra} \textbf{50} (2022), 3932-3939.


\bibitem[GO96]{Granville-Ono}
A. Granville and K. Ono, Defect zero $p$-blocks for finite simple
groups, \emph{Trans. Amer. Math. Soc.} \textbf{348} (1996), 331-347.

\bibitem[GZY22]{gzy}
H. Guan, X. Zhang, and Y. Yang, Recognizing Ree groups  ${}^2G_2(q)$
using the codegree set, \emph {Bull. Austral. Math. Soc.}, to
appear. https://doi.org/10.1017/S0004972722001022


\bibitem[HHN16]{HHN16}
Z. Halasi, C. Hannusch, and H.\,N. Nguyen, The largest character
degrees of the symmetric and alternating groups, \emph{Proc. Amer.
Math. Soc.} \textbf{144} (2016), 1947-1960.


\bibitem[Hig40]{higman40}
G. Higman, The units of group-rings, \emph{Proc. London Math. Soc.}
\textbf{46} (1940), 231-248.

\bibitem[HMT]{hmt}
N.\,N. Hung, A. Moret\'{o}, and P.\,H. Tiep, The codegree
isomorphism problem for finite simple groups II, in preparation.


\bibitem[Hup00]{Huppert}
B. Huppert, Some simple groups which are determined by the set of
their character degrees I, \emph{Illinois J. Math.} \textbf{44}
(2000) 828-842.

\bibitem[Hup06]{Huppert06}
B. Huppert, Some simple groups which are determined by the set of
their character degrees. II, \emph{Rend. Sem. Mat. Univ. Padova}
\textbf{115} (2006), 1-13.

\bibitem[Isa76]{Isaacs}
I.\,M. Isaacs, \emph{Character theory of finite groups}, Academic
Press, 1976.


\bibitem[Isa11]{i}
I.\,M. Isaacs, Element orders and character codegrees, \textit{Arch.
Math.}  \textbf{97} (2011), 499-501.




\bibitem[KLST90]{Kimmerle-et-al}
W. Kimmerle, R. Lyons, R. Sandling, and D.\,N. Teague, Composition
factors from the group ring and Artin's theorem on orders of simple
groups, \emph{Proc. London Math. Soc.} \textbf{60} (1990), 89-122.

\bibitem[Kou20]{Khukhro}
E.\,I. Khukhro and V.\,D. Mazurov, \emph{The Kourovka notebook:
Unsolved problems in group theory}, \textbf{20}th edition.
https://kourovka-notebook.org.

\bibitem[LMT13]{LMT13}
M. Larsen, G. Malle, and P.\,H. Tiep, The largest irreducible
representations of simple groups, \emph{Proc. Lond. Math. Soc.}
\textbf{106} (2013), 65-96.


\bibitem[LY23]{LY22}
 Y. Liu and Y. Yang, Huppert's analogue conjecture for
$\PSL(3,q)$ and $\PSU(3,q)$, \emph{Results Math} \textbf{78}, Paper
No. 7 (2023).

\bibitem[MM21]{MM21}
K. Magaard and G. Malle, Low-dimensional representations of finite
orthogonal groups, \emph{Math. Proc. Cambridge Philos. Soc.}
\textbf{171} (2021), 585-606.

\bibitem[MT11]{malletesterman}
G. Malle and D. Testerman, \emph{Linear algebraic groups and finite
 groups of Lie type},
 Cambridge University Press, Cambridge, 2011.

\bibitem[MZ01]{Malle-Zalesskii}
G. Malle and A.\,E. Zalesskii, Prime power degree representations of
quasi-simple groups, \emph{Arch. Math.} \textbf{77} (2001), 461-468.

\bibitem[Mar22]{Margolis}
L. Margolis, The modular isomorphism problem: A Survey,
\emph{Jahresber. Dtsch. Math.-Ver.} \textbf{124} (2022), 157-196.


\bibitem[McK86]{McK86}
J. McKay, The largest degrees of irreducible characters of the
symmetric group, \emph{Math. Comp.} \textbf{30} (1976), 624-631.

\bibitem[Mic86]{Michler}
G.\,O. Michler, A finite simple group of Lie type has $p$-blocks
with different defects, $p\neq 2$, \emph{J. Algebra} \textbf{104}
(1986), 220-230.

\bibitem[Mih04]{Mih04}
P. Mihailescu, Primary cyclotomic units and a proof of Catalan's
conjecture, \emph{J. Reine Angew. Math.} \textbf{572} (2004),
167-195.

\bibitem[Mor21]{Moreto}
A. Moret\'{o}, Huppert's conjecture for character codegrees,
\emph{Math. Nachr.} \textbf{294} (2021), 2232-2236.

\bibitem[Mor23]{mor23}
A. Moret\'o, Multiplicities of character codegrees of finite groups,
\emph{Bull. London Math. Soc.},  \textbf{55} (2023), 234-241.




\bibitem[NT13]{Navarro-Tiep13}
G. Navarro and P.\,H. Tiep, Characters of relative $p'$-degree over
normal subgroups, \emph{Ann. of Math.} \textbf{178} (2013),
1135-1171.

\bibitem[Ng10]{Nguyen10}
 H.\,N. Nguyen, Low-dimensional complex characters of the symplectic and orthogonal groups,
 \emph{Comm. Algebra} \textbf{38} (2010), 1157-1197.




\bibitem[NT15]{nt}
H.\,N. Nguyen and H.\,P. Tong-Viet, Recognition of finite
quasi-simple groups by the degrees of their irreducible
representations, Groups St Andrews 2013, 439-456, London Math. Soc.
Lecture Note Ser. \textbf{422}, Cambridge Univ. Press, Cambridge,
2015.


\bibitem[QWW07]{Qian}
G. Qian, Y. Wang, and H. Wei, Co-degrees of irreducible characters
in finite groups, \textit{J. Algebra} \textbf{312} (2007), 946-955.

\bibitem[Qia21]{q21}
G. Qian, Element orders and character codegrees, \textit{Bull. Lond.
Math. Soc.} \textbf{53} (2021), 820-824.



\bibitem[RS98]{RS}
U. Riese and P. Schmid, Characters induced from Sylow subgroups,
 \emph{J. Algebra} \textbf{207}
(1998), 682-694.

%
%

\bibitem[SF73]{Simpson-Frame73}
W.\,A. Simpson and J.\,S. Frame, The character tables for $SL(3,q)$,
$SU(3,q^2)$, $PSL(3,q)$, $PSU(3,q^2)$, \emph{Canadian J. Math.}
\textbf{25} (1973), 486-494.

\bibitem[TZ96]{Tiep-Zalesskii96}
P.\,H. Tiep and A.\,E. Zalesskii, Minimal characters of the finite
classical groups, \emph{Comm. Algebra} \textbf{24} (1996),
2093-2167.




\bibitem[VK85]{VK85}
A.\,M. Vershik and S.\,V. Kerov, Asymptotic behavior of the maximum
and generic dimensions of irreducible representations of the
symmetric group (Russian), \emph{Funktsional. Anal. i Prilozhen.}
\textbf{19} (1985), 25-36. English translation: \emph{Functional
Anal. Appl.} \textbf{19} (1985), 21-31.

\bibitem[Wil88]{Willems}
W. Willems, Blocks of defect zero in finite simple groups of Lie
type, \emph{J. Algebra} \textbf{113} (1988), 511-522.


\end{thebibliography}
\end{document}